\documentclass[12pt,reqno]{amsart}

\usepackage{amsmath}
\usepackage{latexsym,amssymb}
\usepackage{mathrsfs}
\usepackage{enumerate}
\usepackage{bm}
\allowdisplaybreaks[1]

\newtheorem{lemma}{Lemma}[section]
\newtheorem{theorem}[lemma]{Theorem}
\newtheorem{thme}[]{Theorem}

\newtheorem{proposition}[lemma]{Proposition}
\newtheorem{corollary}[lemma]{Corollary}
\newtheorem{question}[lemma]{Question}
\theoremstyle{definition}
\newtheorem{definition}[lemma]{Definition}
\newtheorem{example}[lemma]{Example}
\newtheorem{remark}[lemma]{Remark}

\numberwithin{equation}{section}

\newcommand{\bdf}{\begin{definition}}
\newcommand{\edf}{\end{definition}}
\newcommand{\blem}{\begin{lemma}}
\newcommand{\elem}{\end{lemma}}
\newcommand{\bthm}{\begin{theorem}}
\newcommand{\ethm}{\end{theorem}}
\newcommand{\bpf}{\begin{proof}}
\newcommand{\epf}{\end{proof}}
\newcommand{\bprop}{\begin{proposition}}
\newcommand{\eprop}{\end{proposition}}
\newcommand{\bcor}{\begin{corollary}}
\newcommand{\ecor}{\end{corollary}}
\newcommand{\brem}{\begin{remark}}
\newcommand{\erem}{\end{remark}}
\newcommand{\bquest}{\begin{question}}
\newcommand{\equest}{\end{question}}
\newcommand{\bex}{\begin{example}}
\newcommand{\eex}{\end{example}}

\newcommand{\benu}{\begin{enumerate}\renewcommand{\labelenumi}{{\rm (\arabic{enumi})}}\renewcommand{\itemsep}{0pt}}
\newcommand{\eenu}{\end{enumerate}}

\newcommand{\N}{\mathbb{N}}

\newcommand{\R}{\mathbb{R}}
\newcommand{\C}{\mathbb{C}}

\newcommand{\bB}{\mathbb{B}}
\newcommand{\bK}{\mathbb{K}}

\def\cA{\mathcal{A}}

\def\cE{\mathcal{E}}
\def\cF{\mathcal{F}}

\def\cU{\mathcal{U}}

\def\tM{\widetilde{M}}

\DeclareMathOperator{\Ad}{Ad}
\DeclareMathOperator{\Aut}{Aut}

\DeclareMathOperator{\supp}{supp}
\DeclareMathOperator{\tr}{tr}

\newcommand{\stand}{(M, H, J, P)}
\newcommand{\mat}{\mathbb{M}_n}
\newcommand{\e}{\varepsilon}

\def\al{\alpha}
\def\hal{{\hat{\al}}}

\def\be{\beta}

\def\de{\delta}

\def\la{\lambda}

\def\vep{\varepsilon}

\def\ps{{\psi}}

\def\vph{\varphi}
\def\hvph{{\hat{\varphi}}}

\def\om{\omega}
\def\si{\sigma}
\def\ta{\tau}
\def\th{\theta}

\def\De{\Delta}

\def\La{\Lambda}
\def\Ph{\Phi}

\def\Th{\Theta}

\def\col{\colon}

\def\subs{\subset}
\def\ovl{\overline}
\def\oti{\otimes}
\def\rti{\rtimes}
\def\wdt{\widetilde}

\def\loti{\mathbin{\overline{\otimes}}}
\def\hoti{\mathbin{\widehat{\otimes}}}

\DeclareMathOperator{\id}{id}

\begin{document}

\title{Haagerup approximation property \\ for arbitrary von Neumann algebras}

\author[R. Okayasu]{Rui Okayasu$^1$}
\address{$^1$
Department of Mathematics Education, Osaka Kyoiku University,
Osaka \mbox{582-8582},
JAPAN}
\email{rui@cc.osaka-kyoiku.ac.jp}

\author[R. Tomatsu]{Reiji Tomatsu$^2$}
\address{$^2$
Department of Mathematics, Hokkaido University,
Hokkaido \mbox{060-0810},
JAPAN}
\email{tomatsu@math.sci.hokudai.ac.jp}

\thanks{The first author was partially supported by JSPS KAKENHI Grant Number 25800065. The second author was partially supported by JSPS KAKENHI Grant Number 24740095.}

\maketitle

\begin{abstract}
We attempt presenting a notion of the Haagerup approximation property for an arbitrary von Neumann algebra by using its standard form. 
We also prove the expected heredity results for this property. 
\end{abstract}


\section{Introduction}


In the remarkable paper \cite{haa2}, 
U. Haagerup proves that 
the reduced $\mathrm{C}^*$-algebra 
of the non-amenable free group $F_d$ has Grothendieck's metric approximation property. 
He actually shows that 
there exists a sequence of normalized positive definite functions $\varphi_n$ 
on $F_d$ such that 
\begin{itemize}
\item[(a)] $\varphi_n(s)\to 1$ for every $s\in F_d$;
\item[(b)] $\varphi_n$ vanishes at infinity for every $n$. 
\end{itemize}
It is known 
that many classes of locally compact second countable groups 
possess such sequences, 
where pointwise convergence to $1$ is replaced by uniform convergence 
on compact subsets, 
and it is called the {\em Haagerup approximation property}. 
See the book \cite{book} for more details. 

In \cite{cho}, 
M. Choda observes that a countable discrete group $\Gamma$ 
has the Haagerup approximation property
if and only if 
its group von Neumann algebra $L(\Gamma)$ admits 
a sequence of normal contractive completely positive maps $\Phi_n$ 
on $L(\Gamma)$ such that 
\begin{itemize}
\item[(A)]
$\Phi_n\to\mathrm{id}_{L(\Gamma)}$
in the point-ultraweak topology;
\item[(B)] $\tau\circ\Phi_n\leq\tau$ and $\Phi_n$ extends to a compact operator $T_n$ on $\ell^2(\Gamma)$ for every $n$, which is given by
\[
T_n(x\xi_\tau)=\Phi_n(x)\xi_\tau
\ \mbox{for }x\in L(\Gamma),
\]
\end{itemize}
where $\tau$ denotes the canonical tracial state on $L(\Gamma)$.
After her work, 
many authors study the Haagerup approximation property, 
for example, F. Boca \cite{boc}, A. Connes and V. Jones \cite{cj}, 
P. Jolissaint \cite{jol} and S. Popa \cite{pop}. 
However it is defined only for a finite von Neumann algebra. 
In the case of a non-finite von Neumann algebra, 
it is a problem that how to describe {\em vanishing at infinity} 
in (b) or {\em compactness} in (B) for a completely positive map.

After the systematic study of one-parameter family 
of convex cones in the Hilbert space, 
on which a von Neumann algebra acts, 
with a distinguished cyclic and separating vector by H. Araki in \cite{ara}, 
and the independent work by Connes in \cite{co1}, 
Haagerup proves in \cite{haa1} that 
any von Neumann algebra is isomorphic to 
a von Neumann algebra $M$ on a Hilbert space $H$ 
such that there exists a conjugate-linear isometric involution $J$ on $H$ 
and a self-dual positive cone $P$ in $H$ with the following properties:
\begin{itemize}
\item[(i)] $J M J= M'$;
\item[(ii)] $J\xi=\xi$ for any $\xi\in P$;
\item[(iii)] $aJaJ P\subset P$ for any $a\in M$;
\item[(iv)] $JcJ=c^*$ for any $c\in \mathcal{Z}( M):= M\cap M'$.
\end{itemize}
Such a quadruple $\stand$ is called a standard form of the von Neumann algebra $M$. 

Let $\mat$ denote the $n\times n$ complex matrices. 
Then $M\otimes\mat$ operates in its standard form on $H\otimes\mat$ 
with the self-dual positive cone $P^{(n)}$, where $P^{(1)}=P$.  
The partial order on $H\otimes\mat$ induced by $P^{(n)}$ turns 
$H$ into the matrix ordered Hilbert space 
in the sense of M. D. Choi and E. G. Effros in \cite{ce}. 
Thus we will say that an operator $T$ on $H$ is {\em completely positive} 
if $(T\otimes\mathrm{id}_n) P^{(n)}\subset P^{(n)}$ for all $n\geq 1$. 
So for an arbitrary von Neumann algebra $M$,
we give the definition of the Haagerup approximation property 
if the identity of $H$ can be approximated 
in the strong operator topology 
by contractive completely positive compact operators. 

The Haagerup approximation property is also defined 
in other ways for a non-finite von Neumann algebra. 
One definition is the following: 
A $\sigma$-finite von Neumann algebra $M$ 
with a faithful normal state $\varphi$ is said to 
have the Haagerup approximation property for $\varphi$ 
if there exists a net of unital completely positive 
$\varphi$-preserving normal maps $\Phi_n$ on $M$ 
such that 
\begin{itemize}
\item[(A')]
$\Phi_n\to\mathrm{id}_M$ in the point-ultraweak topology;
\item[(B')]
The following implementing operators
$T_n$ on $H_\varphi$ are contractive and compact:
\[
T_n (x\xi_\varphi)=\Phi_n(x)\xi_\varphi\quad \text{for}\ x\in M.
\] 
\end{itemize}
However we wonder whether this definition sufficiently capture
the property of the corresponding compact operator $T_n$ in (B)
in the case where $M$ is finite.
More precisely, one of our main results is the following 
(Theorem \ref{sigma}):

\begin{thme}
Let $M$ be a $\sigma$-finite von Neumann algebra 
with a faithful normal state $\varphi$. 
Then $M$ has the Haagerup approximation property
if and only if 
there exists a net of normal contractive completely positive maps $\Phi_n$ on $M$ such that 
\begin{itemize}
\item[{\rm (A')}]
$\Phi_n\to\mathrm{id}_{M}$ in the point-ultraweak topology;
\item[{\rm (B'')}]
The following implementing operators $T_n$ on $H_\varphi$
are contractive and compact:
\[
T_n(\Delta_\varphi^{1/4}x\xi_\varphi)
=\Delta_\varphi^{1/4}\Phi_n(x)\xi_\varphi
\quad \text{for}\ x\in M.
\] 
\end{itemize}
\end{thme}

In \cite{tor}, 
A. M. Torpe gives a characterization
of semidiscrete von Neumann algebras
in terms of matrix ordered Hilbert spaces. 
Namely a von Neumann algebra $M$
is semidiscrete
if and only if
the identity of the Hilbert space $H$ 
with respect to its standard form 
can be approximated in the strong operator topology 
by completely positive contractions of finite rank. 
A similar characterization of semidiscrete von Neumann algebras 
is also given by M. Junge, Z-J. Ruan and Q. Xu in \cite{jrx}
in terms of non-commutative $L^p$-spaces.
In particular, the non-commutative $L^2$-spaces become
standard forms,
and hence their result is a generalization of her characterization
of semidiscrete von Neumann algebras.
Therefore it immediately follows
that the injectivity
implies the Haagerup approximation property in our sense. 

The Haagerup approximation property has various stabilities.
Among them, we will prove the following result
(Theorem \ref{thm:norm1proj}):

\begin{thme}
Let $N\subs M$ be an inclusion of von Neumann algebras.
Suppose that there exists a norm one projection from $M$ onto $N$.
If $M$ has the Haagerup approximation property, then so does $N$.
\end{thme}

In \cite{CS},
M. Caspers and A. Skalski independently introduce the notion of
the Haagerup approximation property.
Our formulation actually coincides with theirs
because in either case,
the Haagerup approximation property is preserved
under taking the crossed products by $\R$-actions.
(See Remark \ref{rem:CS}.)

This paper is organized as follows: 
In Section 2, the basic notions are reviewed and 
we introduce the Haagerup approximation property for a von Neumann algebra. 
In Section 3, we study some permanence properties 
such as reduced von Neumann algebras, tensor products, 
the commutant and the direct sums. 
In Section 4, we consider the case where 
$M$ is a $\sigma$-finite von Neumann algebra 
with a faithful normal state $\varphi$. 
We present the proof of Theorem A.
We also discuss the free product of von Neumann algebras and examples.
In Section 5,
we study the crossed product of a von Neumann algebra
by a locally compact group.
We show that a von Neumann algebra
has the Haagerup approximation property 
if and only if so does its core von Neumann algebra.
The proof of Theorem B is presented.

\vspace{10pt}
\noindent
{\bf Acknowledgements.}
The authors are grateful
to Narutaka Ozawa for various useful comments on our work.
Theorem B is the answer to his question to us.
The first author would like to thank Marie Choda and Yoshikazu Katayama for fruitful discussions.
The authors also express their gratitude to the referees for several helpful comments and revisions.


\section{Definition}


We first fix notations and recall basic facts.
Let $M$ be a von Neumann algebra. 
We denote by $M_{\mathrm sa}$ and $M^+$, 
the set of all self-adjoint elements 
and all positive elements in $M$, respectively. 
We also denote by $M_*$ and $M_*^+$ 
the space of all normal linear functionals 
and all positive normal linear functionals on $M$, respectively.  

Let us recall the definition of a standard form
of a von Neumann algebra
that is formulated by Haagerup in \cite{haa1}.

\bdf
Let $\stand$ be a quadruple, 
where $M$ is a von Neumann algebra, 
$H$ is a Hilbert space on which $M$ acts, 
$J$ is a conjugate-linear isometry on $H$ with $J^2=1_H$, 
and $P\subset H$ is a closed convex cone which is self-dual, 
i.e., 
\[
P=\{\xi\in H \mid \langle\xi, \eta\rangle\geq 0
\quad \text{for}\ \eta\in P\}.
\]
Then $\stand$ is called a {\em standard form} 
if the following conditions are satisfied: 
\begin{itemize}
\item[(i)]
$J M J= M'$;
\item[(ii)]
$J\xi=\xi$ for any $\xi\in P$;
\item[(iii)]
$xJxJ P\subset P$ for any $x\in M$;
\item[(iv)]
$JcJ=c^*$ for any $c\in \mathcal{Z}(M):= M\cap M'$.
\end{itemize}
\edf

\brem
Recently, Ando and Haagerup prove in \cite[Lemma 3.19]{ah1} 
that the condition (iv) in the above definition actually can be dropped. 
\erem

By the work of Araki \cite{ara},
every functional $\varphi\in M_*^+$
is represented as $\varphi=\omega_{\xi_\varphi}$
by a unique vector $\xi_\varphi\in P$,
where
\[
\omega_{\xi_\varphi}(x)
=\langle x\xi_\varphi, \xi_\varphi\rangle
\quad \text{for}\ x\in M.
\]
Moreover the Araki--Powers--St\o rmer inequality holds:
\[
\|\xi_\varphi-\xi_\psi\|^2
\leq\|\varphi-\psi\|
\leq\|\xi_\varphi-\xi_\psi\|\|\xi_\varphi+\xi_\psi\|
\quad \text{for}\ \varphi, \psi\in M_*.
\]

A vector $\xi\in H$ is said to be {\em self-adjoint}
if $J\xi=\xi$.
We denote by $H_{\mathrm{sa}}$ the set of all self-adjoint vectors in $H$.
For $\xi, \eta\in H_{\mathrm{sa}}$, 
we will write $\xi\geq\eta$ 
if $\xi-\eta\in P$.
Note that for $\xi\in H_{\mathrm{sa}}$ there exist unique vectors $\xi_+,\xi_-\in P$
such that $\xi=\xi_+-\xi_-$ and $\langle \xi_+,\xi_-\rangle=0$.

We next introduce that a faithful normal semifinite (f.n.s.)\ weight gives a standard form. 
We refer readers
to the book of Takesaki \cite{t2}
for details. 

Let $\varphi$ be an f.n.s.\ weight on a von Neumann algebra $M$ and let 
\[
n_\varphi
:=\{x\in M \mid \varphi(x^*x)<\infty\}.
\]
Then $H_\varphi$ is the completion of $n_\varphi$ with respect to the norm 
\[
\|x\|_\varphi^2
:=\varphi(x^*x)\quad \text{for}\ x\in n_\vph.
\]
We write the canonical injection 
$\Lambda_\varphi\colon n_\varphi\to H_\varphi$. 

Then 
\[
\mathcal{A}_\varphi
:=\Lambda_\varphi(n_\varphi\cap n_\varphi^*)
\] 
is an achieved left Hilbert algebra with the multiplication 
\[
\Lambda_\varphi(x)\cdot\Lambda_\varphi(x)
:=\Lambda_\varphi(xy)\quad \text{for}\ x\in n_\varphi\cap n_\varphi^*
\]
and the involution 
\[
\Lambda_\varphi(x)^\sharp
:=\Lambda_\varphi(x^*)\quad \text{for}\ x\in n_\varphi\cap n_\varphi^*.
\]
Let $\pi_\varphi$ be the corresponding representation of $M$ on $H_\varphi$. 
We always identify $M$ with $\pi_\varphi(M)$. 

Let $S_\varphi$ be the closure of the conjugate-linear operator $\xi\mapsto\xi^\sharp$ on $H_\varphi$, 
which has the polar decomposition 
\[
S_\varphi
=J_\varphi\Delta_\varphi^{1/2},
\]
where $J_\varphi$ is the modular conjugation 
and $\Delta_\varphi$ is the modular operator. 
Then we have a self-dual positive cone
\[
P_\varphi
:=\overline{\{\xi(J_\varphi\xi) \mid \xi\in\mathcal{A}_\varphi\}}\subset H_\varphi.
\]
Note that $P_\vph$ is given by the closure of the set
of $\La_\vph(x\si_{i/2}^\vph(x)^*)$,
where $x\in \cA_\vph$ is entire
with respect
to the modular automorphism group $\si_t^\vph:=\Ad \De_\vph^{it}|_M$.

Therefore
the quadruple $(M, H_\varphi, J_\varphi, P_\varphi)$
is a standard form. 
A standard form is, in fact, unique
up to a spatial isomorphism,
and so it is independent to the choice
of an f.n.s.\ weight $\varphi$.

\bthm[{\cite[Theorem 2.3]{haa1}}]\label{unique}
Let $(M_1, H_1, J_1, P_1)$ and $(M_2, H_2, J_2, P_2)$ be two standard forms 
and let $\pi\colon M_1\to M_2$ be an isomorphism. 
Then there exists a unique unitary $u\colon H_1\to H_2$ such that
\benu
\item $\pi(x)=uxu^*$ for any $x\in M_1$;
\item $J_2=uJ_1u^*$;
\item $P_2=uP_1$.
\eenu
\ethm

Let us consider the $n\times n$ matrix algebra $\mat$ 
with the normalized trace $\mathrm{tr}_n$. 
If we define the inner product on $\mat$ by 
\[
\langle x, y\rangle
:=\mathrm{tr}_n(y^*x)
\quad \text{for}\ x, y\in\mat,
\]
then the algebra $\mat$ can be also regarded as a Hilbert space. 
Moreover $\mat$ is an achieved left Hilbert algebra 
such that the modular operator is the identity operator on $\mat$
and the modular conjugation is 
the canonical involution $J_{\mathrm{tr}_n}\colon x\mapsto x^*$. 
Hence the quadruple $(\mat, \mat, J_{\mathrm{tr}_n}, \mat^+)$ 
is a standard form.

Let $\stand$ be a standard form. 
Next we consider the von Neumann algebra $\mat(M):=M\otimes\mat$ 
on $\mat(H):=H\otimes\mat$. 
If we consider an f.n.s.\ weight $\varphi\otimes\mathrm{tr}_n$ 
on $M\otimes\mat$ 
for a fixed f.n.s.\ weight $\varphi$ on $M$, 
then we can give a standard form of $\mat(M)$ as mentioned before.  
However we give a standard form without using an f.n.s.\ weight. 
The following definition is given by Miura and Tomiyama in \cite{mt}.

\bdf[{\cite[Definition 2.1]{mt}}]
Let $\stand$ be a standard form and $n\in\N$. 
A matrix $[\xi_{i, j}]\in\mat(H)$ is said to be {\em positive} if 
\[
\sum_{i, j=1}^nx_iJx_jJ\xi_{i, j}\in P
\quad \text{for all }\ x_1, \dots, x_n\in M.
\]
We denote by $P^{(n)}$ the set of all positive matrices 
$[\xi_{i, j}]$ in $\mat(H)$. 
\edf

\bprop[{\cite[Proposition 2.4]{mt}}, {\cite[Lamma 1.1]{sw1}}]
Let $\stand$ be a standard form and $n\in\N$. 
Then $(\mat(M), \mat(H), J^{(n)}, P^{(n)})$ is a standard form, 
where $J^{(n)}:=J\otimes J_{\mathrm{tr}_n}$. 
\eprop

\bdf
Let $(M_1, H_1, J_1, P_1)$ and $(M_2, H_2, J_2, P_2)$ be two standard forms. 
We will say that a bounded linear (or conjugate-linear) operator 
$T\colon H_1\to H_2$ is {\em $n$-positive} 
if  
\[
T^{(n)}P_1^{(n)}\subset P_2^{(n)},
\]
where $T^{(n)}\colon \mat(H_1)\to \mat(H_2)$ is defined by 
\[
T^{(n)}([\xi_{i, j}]):=[T\xi_{i, j}].
\] 
Moreover we will say that $T$ is {\em completely positive}
({\em c.p.}) 
if $T$ is $n$-positive for any $n\in\N$,
\edf

We are now ready to give our definition of
the Haagerup approximation property for a von Neumann algebra.

\bdf\label{defn:HAP}
A W$^*$-algebra $M$
has the {\em Haagerup approximation property} (HAP)
if there exists a standard form $(M, H, J, P)$
and a net of contractive completely positive (c.c.p.)
compact operators $T_n$ on $H$ 
such that $T_n\to 1_{H}$ in the strong topology.
\edf

From this definition, it is clear that
if a von Neumann algebra $M_2$ is isomorphic to
$M_1$ which has the HAP,
then so does $M_1$.
Moreover it does not depend on the choice of a standard form.
Indeed,
let $(M_1, H_1, J_1, P_1)$ and $(M_2, H_2, J_2, P_2)$ be two standard forms
of von Neumann algebras, 
and $\pi\colon M_1\to M_2$ be an isomorphism.
By Theorem \ref{unique}, 
there is a unitary $u\colon H_1\to H_2$ 
such that 
$\pi(x)=uxu^*$ for $x\in M_1$,
$J_2=uJ_1u^*$,
and $P_2=uP_1$.
Let $T_n^1$ be a net of c.c.p.\ compact operators
on $H_1$ as in the previous definition.
Then one can easily check 
that $T_n^2:=uT_n^1 u^*$ gives a desired net 
of c.c.p.\ compact operators on $H_2$.
 
%
%
%
%
%

\brem
A notion of the HAP can be also defined for a matrix ordered Hilbert space 
in the sense of Choi and Effros in \cite{ce}. 
However
we only consider the case of a standard form
of a von Neumann algebra in this paper.
\erem

In \cite{tor},
Torpe gives a characterization of semidiscrete von Neumann algebras 
in terms of standard forms. 
In \cite{jrx}, 
Junge, Ruan and Xu also give a similar characterization 
of semidiscrete von Neumann algebras 
in terms of non-commutative $L^p$-spaces for $1\leq p<\infty$. 
In particular, in the case where $p=2$, 
the non-commutative $L^2$-space gives a standard form.
Hence their result is a generalization of her characterization.
As a corollary, 
the injectivity implies the HAP.

\bthm[{\cite[Theorem 2.1]{tor}, \cite[Theorem 3.2]{jrx}}]\label{tor}
Let $\stand$ be a standard form. 
Then the following are equivalent:
\benu
\item $M$ is semidiscrete;
\item There exists a net of c.c.p.\ finite rank operators $T_n$ on $H$ 
such that $T_n\to 1_H$ in the strong topology.
\eenu
\ethm

\bcor\label{cor:tor}
If a von Neumann algebra $M$ is injective, then $M$ has the HAP. 
\ecor

\brem
Unfortunately, Torpe's paper \cite{tor} is unpublished. 
However the implication (1) $\Rightarrow$ (2) is proved by L. M. Schmitt 
in \cite{sch} with her techniques. 
We also remark her proof of the other implication in Remark \ref{torperem}.
\erem


\section{Permanence properties}


In this section, 
we study various permanence properties of the Haagerup approximation property.


\subsection{Reduction}


We first recall the following result in \cite{haa1}.

\blem[{\cite[Corollary 2.5, Lemma 2.6]{haa1}}]\label{reducestand}
Let $\stand$ be a standard form of a von Neumann algebra
and $q$ a projection of the form $q=pJpJ$, 
where $p\in M$ is a projection.
\benu
\item The induction $pap\mapsto qxq$
is an isomorphism from $pMp$ onto $qMq$;
\item The quadruple $(qMq, qH, qJq, qP)$ is a standard form.
\eenu
\elem

Let $\stand$ be a standard form 
and $p\in M$ be a projection 
with $q:=pJpJ$. 
We write $M_q:=qMq$, $H_q:=qH$, $J_q:=qJq$ and $P_q:=qP$, respectively. 
On the one hand, we have a standard form 
\[
(\mat(M_q), \mat(H_q), J_q^{(n)}, P_q^{(n)}).
\]

Notice that $(\mat(M), \mat(H), J^{(n)}, P^{(n)})$ is a standard form. 
Set $p^{(n)}:=p\otimes 1_n\in\mat(M)$ 
and $q^{(n)}:=p^{(n)}J^{(n)}p^{(n)}J^{(n)}$. 
Then 
\[
q^{(n)}
:=p^{(n)}J^{(n)}p^{(n)}J^{(n)}
=q\otimes 1_n.
\] 
On the other hand, by Lemma \ref{reducestand}, we have a standard form
\[
(q^{(n)}\mat(M)q^{(n)}, q^{(n)}\mat(H), q^{(n)}J^{(n)}q^{(n)}, q^{(n)}P^{(n)}).
\] 
Note that $\mat(M_q)=q^{(n)}\mat(M)q^{(n)}$, $\mat(H_q)=q^{(n)}\mat(H)$ 
and $J_q^{(n)}=q^{(n)}J^{(n)}q^{(n)}$. 
Moreover two standard forms, in fact, coincide.

\blem\label{mat&proj}
In the above setting, $P_q^{(n)}=q^{(n)} P^{(n)}$.
\elem

\bpf
Let $[\xi_{i, j}]\in P^{(n)}$. 
For any $x_1, \dots, x_n\in  M$, we have
\[
\sum_{i, j=1}^n(qx_iq)(qJq)(qx_jq)(qJq)(q\xi_{i, j})
=q\sum_{i, j=1}^n(px_ip)J(px_jp)J\xi_{i, j}\in q P.
\]
Hence $[q\xi_{i, j}]\in P_q^{(n)}$. 
Therefore $q^{(n)} P^{(n)}\subset  P_q^{(n)}$. 

Next we will show that $P_q^{(n)}\subset q^{(n)} P^{(n)}$. 
Let $\xi\in  P_q^{(n)}$. 
Then $\omega_\xi\in \mat(M_q)_*^+$. 
Since $q^{(n)} P^{(n)}$ is a self-dual cone of a standard form
of $\mat(M_q)$,
there exists $\eta\in q^{(n)} P^{(n)}$ 
such that $\omega_\xi=\omega_\eta$ in $\mat(M_q)_*^+$. 
By the discussion above, we also have $\eta\in  P_q^{(n)}$. 
By the uniqueness of $\xi$, 
we have $\xi=\eta\in  q^{(n)} P^{(n)}$. 
Therefore $ P_q^{(n)}=q^{(n)} P^{(n)}$.
\epf

\blem\label{ad}
For $x\in M$, $xJxJ$ is a c.p.\ operator. 
\elem

\bpf
For $[\xi_{i, j}]\in P^{(n)}$, 
we have 
\[
[xJxJ\xi_{i, j}]
=(x\otimes 1_n)(J\otimes J_{\mathrm{tr}_n})(x\otimes 1_n)(J\otimes J_{\mathrm{tr}_n})[\xi_{i, j}]
\in P^{(n)}.
\]
\epf

\bthm\label{reduce}
Let $\stand$ be a standard form 
and $p\in M$ a projection 
with $q:=pJpJ$. 
If $M$ has the HAP, 
then so does $qMq$. 
In particular, $pMp$ also has the HAP. 
\ethm

\bpf
Since $M$ has the HAP, 
there exists a net of c.c.p.\ compact operators $T_n$ on $H$ 
such that $T_n\to 1_{H}$ in the strong topology. 
Then $S_n:=qT_n q$ gives a desired net for $qMq$ by Lemma \ref{ad}. 
By Lemma \ref{reducestand},
$pMp$ is isomorphic to $qMq$.
Hence $pMp$ also has the HAP.
\epf

\bprop\label{increasing}
Let $\stand$ be a standard form 
and $(p_n)$ an increasing net of projections of $M$
such that $p_n\to 1_H$ in the strong operator topology. 
If $p_n  M p_n$ has the HAP for all $n$, 
then so does $M$.
\eprop

\bpf
Let $q_n:=p_n Jp_n J$. 
By Lemma \ref{reducestand},
$q_n M q_n$ has the HAP for all $n$. 
Let $F$ be a finite subset of $H$ and $\e>0$. 
Since $q_n\to 1$ in the strong topology, 
there exists $n_F$ such that 
\[
\|q_{n_F}\xi-\xi\|<\e/2\quad \text{for}\ \xi\in F.
\] 
Since $q_{n_F} M q_{n_F}$ has the HAP, 
there exists a c.c.p.\ compact operator $T$ on $q_{n_F}H$ 
such that 
\[
\|T(q_{n_F}\xi)-q_{n_F}\xi\|<\e/2\quad \text{for}\ \xi\in F.
\] 

Now we define a c.c.p.\ compact operator $S:=Tq_{n_F}$ on $H$. 
Since
\[
\|S\xi-\xi\|
\leq\|T(q_{n_F}\xi)-q_{n_F}\xi\|+\|q_{n_F}\xi-\xi\|
<\e
\quad \text{for}\ \xi\in F.
\]
So $M$ has the HAP. 
\epf


\subsection{Norm one projection}


Secondly, we consider an inclusion of von Neumann algebras,
$N\subs M$
and study when $N$ inherits the HAP from $M$.
One answer will be presented
in Theorem \ref{thm:norm1proj},
which states that it is the case
when there exists a norm one projection from $M$
onto $N$.
In the following, let us prove this assuming
normality.

\bthm
\label{thm:expectation}
Let $N\subset M$ be an inclusion of von Neumann algebras.
Suppose that
there exists a normal conditional expectation from $M$ onto $N$.
If $M$ has the HAP, then so does $N$.
\ethm

\bpf
Let $\cE\col M\to N$ be a normal conditional expectation.
Take an increasing net of
$\sigma$-finite projections $p_n$ in $N$
such that $p_n\to 1$ in the strong topology.
Then we have a normal conditional expectation
$\cE_n\colon p_n M p_n\to p_n N p_n$, 
which is given by
$\cE_n(p_n xp_n)
:=\cE(p_nxp_n)=p_n\cE(x)p_n$
for
$x\in M$.
By Theorem \ref{reduce}, 
$p_n M p_n$ has the HAP. 
Thanks to Proposition \ref{increasing},
$N$ has the HAP if each $p_n N p_n$ does.
Hence we may and do assume that $N$ is $\si$-finite.

Suppose that $\cE$ is faithful.
Let $\ps$ be a faithful normal state on $N$.
Then $\varphi:=\psi\circ\cE\in M_*^+$ is also faithful. 
%
%
%
The projection $E$ from $H_\vph$
onto $K_\vph:=\ovl{N\xi_\vph}$, 
is given by
$E(x\xi_\varphi)=\cE(x)\xi_\varphi$
for $x\in M$.

Thanks to \cite[IX \textsection 4 Theorem 4.2]{t2}, 
the modular operator $\Delta_\varphi$ and $E$ commute.
Thus it turns out that $E\col H_\vph\to K_\vph$
is a c.p.\ operator,
where we regard $K_\vph$ as the GNS Hilbert space of $N$
with respect to $\ps$.
Moreover the inclusion operator $V\colon K_\vph\to H_\vph$
is also a c.p.\ operator.
Let $T_n$ be a net of c.c.p.\ compact operators
for $M$ such that
$T_n\to 1_H$ in the strong topology.
Then $ET_n V$ gives a net of c.c.p.\ compact operators 
such that $ET_n V\to 1_K$ in the strong topology,
that is, $N$ has the HAP. 

In the case that $\cE$ is not faithful,
there exists a projection $e\in M\cap N'$ 
such that the central support of $e$ in $N'$ is the identity 
and
$\{x\in M \mid \cE(x^*x)=0\}=M(1-e)$.
Moreover we obtain a faithful normal conditional expectation
$\cE'\colon  eMe\to Ne$, 
which is given by $\cE'(x)=\cE(x)e$ for $x\in eMe$. 
By Theorem \ref{reduce}, 
$eMe$ has the HAP, 
and so does $Ne$ by our discussion above. 
Let $x\in N$.
Then $\cE(xe)=x$, which implies that
$\cE$ is an isomorphism from $Ne$ onto $N$,
and $N$ has the HAP. 
\epf


\subsection{Tensor product and commutant}


Next we show the following theorem on tensor products.

\bthm\label{tensor}
Let $M_1$ and $M_2$ be von Neumann algebras.
Then $M_1$ and $M_2$ have the HAP
if and only if
so does $M_1\mathbin{\overline{\otimes}}M_2$. 
\ethm

To prove this,
we introduce several results 
from \cite{mt,sw1,sw2}. 
Let $(M_1, H_1, J_1, P_1)$
and $(M_2, H_2, J_2, P_2)$
be two standard forms of von Neumann algebras. 
For $\zeta\in H_1\otimes H_2$, 
we define a bounded conjugate-linear map $r(\zeta)\colon H_1\to H_2$ by
\[
r(\zeta)(\xi)
:=
(\xi^*\oti1)\zeta
\quad \text{for}\ \xi\in H_1.
\]

\bdf[{\cite[Definition 2.7]{mt}}]
For $n\in\N$, 
the set of all elements $\zeta\in H_1\otimes H_2$ 
such that $r(\zeta)$ is a c.p.\ map from $H_1$ to $H_2$ 
is denote by $P_1\hoti P_2$. 
\edf

\bthm[{\cite[Theorem 2.8]{mt}}, {\cite[Theorem 1]{sw2}}]
The cone $P_1\hoti P_2$ contains $P_1\otimes P_2$ 
and is the self-dual cone in $H_1\otimes H_2$ 
such that $(M_1\loti M_2, H_1\otimes H_2, J_1\otimes J_2, P_1\hoti P_2)$ is a standard form. 
\ethm

\bcor[{\cite[Corollary 2.9]{mt}}]\label{mt}
The cone $P_1\hoti P_2$ coincides with the closure of 
\[
\Big{\{}
\sum_{i, j=1}^n\xi_{i, j}\otimes\eta_{i, j} \mid n\in\N,
\
[\xi_{i, j}]\in P_1^{(n)},
\
[\eta_{i, j}]\in P_2^{(n)}
\Big{\}}.
\]
\ecor

Under the identification
$\mat(M_1\loti M_2)
=M_1\loti \mat(M_2)$ and $\mat(H_1\otimes H_2)=H_1\otimes \mat(H_2)$, 
the self-dual positive cone $P_1\hoti P_2^{(n)}$ gives a standard form of $\mat(M_1\loti M_2)$ by \cite[Corollary 2.3]{sw1}.

\blem\label{cptensor}
If $T_1$ and $T_2$ are c.p.\ operators on $H_1$ and $H_2$, respectively, 
then $T_1\otimes T_2$ is a c.p.\ operator on $H_1\otimes H_2$. 
\elem

\bpf
Since $T_1$ and $T_2$ are c.p.\ operators, 
it suffices to show that $T_1\otimes T_2$ is positive. 
Let $\zeta\in P_1\hoti P_2$. 
By Corollary \ref{mt}, we may assume that 
\[
\zeta
=\sum_{i, j=1}^n\xi_{i, j}\otimes\eta_{i, j},
\]
where $n\in\N$, $[\xi_{i, j}]\in P_1^{(n)}$, $[\eta_{i, j}]\in P_2^{(n)}$. 
Then 
\[
(T_1\otimes T_2)\zeta
=\sum_{i, j=1}^nT_1\xi_{i, j}\otimes T_2\eta_{i, j},
\]
which belongs to $P_1\hoti P_2$ by Corollary \ref{mt}.
\epf

\bpf[Proof of Theorem \ref{tensor}]
We show the ``only if'' part.
Since $M_i$ has the HAP, 
there exists a net of c.c.p.\ compact operators $T_n^{i}$ on $H_i$ 
such that $T_n^{i}\to 1_{H_i}$ in the strong topology for $i=1, 2$. 
Then by Lemma \ref{cptensor}, 
$T_n:=T_n^{1}\otimes T_n^{2}$ gives a desired net of c.c.p.\ compact operators on $H_1\otimes H_2$.
The ``if'' part follows from Theorem \ref{thm:expectation}
with slice maps by states.
\epf

The proof of the following theorem is inspired by \cite[Theorem 2.8]{ht}.

\bthm\label{commutant}
If $M$ has the HAP,
then $M'$ has the HAP.
\ethm

\bpf
%
Since a representation of a von Neumann algebra
consists of an amplification,
an induction and a spatial isomorphism,
it suffices to prove the statement for
$N=M\oti 1_K$ or $Q=Mp'$ for a projection $p'\in M'$,
where $K$ denotes a Hilbert space.
Taking the commutants of these,
we obtain $N'= M'\loti\bB(K)$
or $Q'=p' M'p'$.
They have the HAP by Theorem \ref{reduce} and Theorem \ref{tensor}.
\epf
%
%
%

\bcor\label{lem:central1}
Let $M$ be a von Neumann algebra 
and $p\in M$ be a projection 
with central support $1$ in $M$.
The von Neumann algebra $M$ has the HAP 
if and only if 
$p M p$ has the HAP.
In particular,
a factor $M$ has the HAP if and only if
a corner of $M$ has the HAP.
\ecor

\bpf
The ``only if'' part is nothing but
Theorem \ref{reduce}.
We will show the ``if'' part.
Suppose that $pMp$ has the HAP.
Then by Theorem \ref{commutant},
$(pMp)'=M'p$ has the HAP.
Since the central support of $p$ in $M'$ equals 1,
the induction $M'\ni x\mapsto xp\in M'p$
is an isomorphism.
Thus $M'$ has the HAP, and so does $M$
again by Theorem \ref{commutant}.
\epf


\subsection{Direct sum}


Finally, this section concludes 
by considering the direct sum of von Neumann algebras.

\bthm\label{sum}
Let $(M_i)_{i\in I}$ be a family of von Neumann algebras. 
Then $\bigoplus_{i\in I} M_i$ has the HAP 
if and only if 
$M_i$ has the HAP for all $i\in I$. 
\ethm

\bpf
We write $M:=\bigoplus_{i\in I} M_i$. 
If $M$ has the HAP, 
then $M_i$ has the HAP by Theorem \ref{reduce}. 

Conversely, let $(M_i, H_i, J_i, P_i)$ be a standard form 
for $i\in I$. 
We denote 
\[
H:=\bigoplus_{i\in I} H_i, 
J:=\bigoplus_{i\in I} J_i,  
P:=\bigoplus_{i\in I} P_i.
\]
Then $(M, H, J, P)$ is a standard form. 
Let $F$ be a subset of $I$, 
and $T_i$ be a c.c.p.\ compact operator on $H_i$ for $i\in I$. 
Then we define a c.c.p.\ compact operator $T_F$ on $H$ by
\[
T_F:=(\bigoplus_{i\in F}T_i)p_FJp_FJ,
\]
where $p_F$ is the projection of $M$ onto $\bigoplus_{i\in F} M_i$. 

Let $\e>0$ and $\xi^1, \dots, \xi^m\in H$. 
We denote $\xi^k=\bigoplus_{i\in I}\xi_i^k$ 
with $\xi_i^k\in H_i$ for $1\leq k\leq m$. 
Since $\|\xi^k\|^2=\sum_{i\in I}\|\xi_i^k\|^2<\infty$, 
there is a finite subset $F\subset I$ 
such that 
\[
\sum_{i\not\in F}\|\xi_i^k\|^2
<\frac{\, \e\, }{2}
\quad \text{for}\ 1\leq k\leq m.
\] 
For each $i\in F$, 
since $ M_i$ has the HAP, 
there exists a c.c.p.\ compact operator $T_i$ on $H_i$ 
such that 
\[
\|T_i\xi_i^k-\xi_i^k\|^2
<\frac{\e}{2|F|}
\quad \text{for}\ 1\leq k\leq m.
\] 
Then 
\[
\|T_F\xi^k-\xi^k\|^2
=\sum_{i\in F}\|T_i\xi_i^k-\xi_i^k\|^2
	+\sum_{i\not\in F}\|\xi_i^k\|^2
<\e.
\]
\epf

\bcor\label{surjection}
Let $\pi$ be a normal $*$-homomorphism from $M$
into $N$. 
Then $M$ has the HAP 
if and only if 
$\pi(M)$ and $\ker\pi$ have the HAP. 
\ecor

\bpf
Take a central projection $z\in M$ 
such that $\ker\pi=Mz$ and $M(1-z)$ is isomorphic to $\pi(M)$. 
Since $ M=Mz\oplus M(1-z)$,
the corollary follows from Theorem \ref{sum}.
\epf

%
%
%
%
%

%
%
%
%




\section{$\sigma$-finite von Neumann algebras}


Let $M$ be a $\sigma$-finite von Neumann algebra 
with a faithful state $\varphi\in M_*^+$. 
We denote by $(\pi_\varphi, H_\varphi, \xi_\varphi)$ 
the GNS construction of $(M, \varphi)$. 
We always identify $M$ with $\pi_\varphi(M)$.
We also denote by $\Delta_\varphi$ and $J_\varphi$ 
the modular operator and the modular conjugation, respectively. 
Denote by $P_\varphi$ the norm closure of the cone
$\Delta_\varphi^{1/4} M^+\xi_\varphi$ in $H_\varphi$.
Then $(M,H_\vph,J_\vph,P_\vph)$ is a standard form.


\subsection{Construction of completely positive maps}


Let $\stand$ be a standard form 
and $\xi_0\in P$ be a cyclic and separating vector.
Then we denote by $\De_{\xi_0}$ the associated modular operator.
Note that the associated
modular conjugation equals $J$ by \cite[Lemma 2.9]{haa1}.

\blem[cf.\ {\cite[Theorem 2.7]{co1}}, {\cite[Lemma 4.8]{ah2}}]\label{co}
Let $\stand$ be a standard form of a $\sigma$-finite von Neumann algebra $M$. 
Let $\xi_0\in P$ be a cyclic and separating vector. 
Then the map $\Theta_{\xi_0}\colon  M\to  H$, which is defined by 
\[
\Theta_{\xi_0}(x):=\Delta_{\xi_0}^{1/4}x\xi_0\quad \text{for}\ x\in M,
\] 
induces an order isomorphism 
between $\{x\in M_{\mathrm{sa}} \mid -c 1\leq x\leq c 1\}$ 
and $K_{\xi_0}:=\{\xi\in H_{\mathrm{sa}} \mid -c\xi_0\leq \xi\leq c\xi_0\}$ 
for each $c>0$.  
Moreover $\Theta_{\xi_0}$ is $\sigma(M, M_*)$ - $\sigma(H, H)$ continuous.
\elem

\bpf
The first part of the lemma is proved in \cite[Lemma 4.8]{ah2}. 
We need to show that $\Theta_{\xi_0}$ is $\sigma(M, M_*)$ - $\sigma(H, H)$ continuous. 
Since 
\[
\Delta_{\xi_0}^{1/4}x\xi_0
=(\Delta_{\xi_0}^{1/4}+\Delta_{\xi_0}^{-1/4})^{-1}(x\xi_0+J_{\xi_0}x^*\xi_0)
\]
and $(\Delta_{\xi_0}^{1/4}+\Delta_{\xi_0}^{-1/4})^{-1}$ is bounded, 
it follows that $\Theta_{\xi_0}$ is $\sigma(M, M_*)$ - $\sigma(H, H)$ continuous. 
\epf

\blem\label{cpex}
Let $\stand$ be a standard form and $\xi\in P$. Then 
\benu
\item A functional $f_\xi\colon H\to\C,
\zeta\mapsto\langle\zeta, \xi\rangle$, is a c.p.\ operator;
\item An operator $g_\xi\colon\C\to H, z\mapsto z\xi$,
is a c.p.\ operator.
\eenu
\elem

\bpf
(1) For $[\xi_{i, j}]\in P^{(n)}$, we have
\[
f_\xi^{(n)}([\xi_{i, j}])=[f_\xi(\xi_{i, j})]=[\langle \xi_{i, j}, \xi\rangle].
\]
This is a positive matrix.
Indeed, if $z_1,\dots,z_n\in\C$, then
\[
\sum_{i,j=1}^n
\langle \xi_{i, j}, \xi\rangle z_i \ovl{z_j}
=
\langle \sum_{i,j=1}^nz_i Jz_j J\xi_{i,j},\xi\rangle
\geq0.
\]

(2)
Let $[z_{i, j}]\in\mat^+$.
Take $[w_{i,j}]\in\mat$ so that
$z_{i,j}=\sum_{k=1}^nw_{i,k}\ovl{w_{j,k}}$.
Then $g_\xi^{(n)}([z_{i, j}])=[z_{i, j}\xi]$ belongs to $P^{(n)}$.
Indeed, for $x_1, \dots, x_n\in M$,
putting $y_k:=\sum_{i=1}^n x_i w_{i,k}$, we have
\[
\sum_{i,j=1}^n x_i Jx_j J z_{i, j}\xi
=
\sum_{k=1}^n y_kJy_k J \xi\in P.
\]
\epf

\blem\label{cyclic}
Suppose that 
there exists a net of 
c.p.\ operators $S_n$ on $H_\varphi$
such that $S_n\to 1_{H_\varphi}$ in the strong topology.
Then there exists a net of
c.p.\ operators $S'_n$ on $H_\varphi$ 
satisfying the following:
\benu
\item $S'_n\to 1_{H_\varphi}$ in the strong topology;
\item $S'_n-S_n$ has rank one for all $n$;
\item $\|S'_n-S_n\|\to 0$;
\item $S'_n\xi_\varphi$ is cyclic and separating for all $n$.
\eenu
In particular, if $M$ has the HAP, then
there exists a net of c.c.p.\ compact operators 
$T_n$ on $H_\varphi$ such that 
$T_n\to 1_{H_\varphi}$ 
in the strong topology
and $T_n\xi_\varphi$ 
is cyclic and separating for all $n$.
\elem

\bpf
Let $(S_n)$ be a net of c.p.\ operators 
on $H_\varphi$ 
such that $S_n\to 1_{H_\varphi}$ in the strong topology. 
Set $\eta_n:=S_n\xi_\varphi\in P_\varphi$. 
Then we define $\xi_n:=\eta_n+(\eta_n-\xi_\varphi)_-\in P_\varphi$. 
Since
\[
\xi_n-\xi_\varphi
=\eta_n+(\eta_n-\xi_\varphi)_--\xi_\varphi
=(\eta_n-\xi_\varphi)_+\in P_\varphi,
\]
we have $\xi_n\geq\xi_\varphi$. 
For any $\eta\in P_\varphi$, 
if $\langle \xi_n, \eta\rangle=0$, 
then $\langle \xi_\varphi, \eta\rangle=0$, 
and thus $\eta=0$. 
By \cite[Lemma 4.3]{co1}, 
$\xi_n$ is cyclic and separating. 

Now we define a bounded operator $S'_n$ on $H_\varphi$ by
\[
S'_n\xi
:=S_n\xi+\langle \xi, \xi_\varphi\rangle(\xi_n-\eta_n)
\quad\text{for}\
\xi\in H_\varphi.
\]
By Lemma \ref{cpex}, 
$S'_n$ is a c.p.\ operator. 
Note that $S'_n-S_n$ has rank one and  
\[
S'_n\xi_\varphi
=S_n\xi_\varphi+\langle \xi_\varphi, \xi_\varphi\rangle(\xi_n-\eta_n)
=\xi_n.
\]
Since 
\[
\|\xi_n-\eta_n\|
=\|(\eta_n-\xi_\varphi)_-\|\leq \|\eta_n-\xi_\varphi\|
=\|S_n\xi_\varphi-\xi_\varphi\|\to 0,
\]
we have $\|S_n'\xi-\xi\|\to 0$ for any $\xi\in H_\varphi$, 
and $\|S_n'-S_n\|\to 0$. 

If $M$ has the HAP, 
then we may assume that 
the above operators $S_n$ are compact with $\|S_n\|\leq 1$.
Let $\xi\in H_\varphi$ with $\|\xi\|=1$. 
Since $\|S_n\|\leq 1$, 
we obtain 
\[
0
\leq 1-\|S_n\|
\leq \|\xi\|-\|S_n\xi\|
\leq \|\xi-S_n\xi\|
\to 0.
\]
Namely $\|S_n\|\to 1$, and thus $\|S'_n\|\to 1$. 
Then $T_n:=\|S_n'\|^{-1}S'_n$ is a c.c.p.\ compact operator 
such that $T_n\to\mathrm{id}_{H_\varphi}$ in the strong topology, 
and $T_n\xi_\varphi$ is cyclic and separating.
\epf

\blem[cf.\ {\cite[Theorem 10]{ara}}]\label{ara}
Let $\stand$ be a standard form and 
$\xi_0\in P$ be cyclic and separating vector.
If $(\xi_n)$ is a net of cyclic and separating vectors in $P$
such that $\xi_n\to\xi_0$, 
then $f(\Delta_{\xi_n})\to f(\Delta_{\xi_0})$ in the strong topology 
for any $f\in C_0[0,\infty)$. 
In particular $(\Delta_{\xi_n}^{1/4}+\Delta_{\xi_n}^{-1/4})^{-1}\to(\Delta_{\xi_0}^{1/4}+\Delta_{\xi_0}^{-1/4})^{-1} $ in the strong topology.
\elem

\blem[cf.\ {\cite[Theorem 1.1]{w}}]\label{wo}
Let $\stand$ be a standard form
and $\xi_0\in P$ be a cyclic and separating vector.
Let $C>0$ 
and $s$ be a positive sesquilinear form on $M\times  M$ 
such that $s(x, y)\geq 0$ 
and $s(x, 1)\leq C\omega_{\xi_0}(x)$ 
for $x, y\in  M^+$. 
Then
\[
s(x, x)
\leq C\|\Delta_{\xi_0}^{1/4}x\xi_0\|^2
\quad \text{for}\ x\in  M.
\]
\elem

\blem\label{lem:etaab}
Let $\stand$ be a standard form
and $\eta_0\in P$ be a cyclic and separating vector.
Then for $x, y\in M^+$,
one has
\[
0\leq
\langle \De_{\eta_0}^{1/4}x\eta_0,\De_{\eta_0}^{1/4}y\eta_0\rangle
\leq \|y\|\langle x\eta_0, \eta_0\rangle.
\]
\elem
\bpf
Put $y':=JyJ\in M'$.
Then we have
\begin{align*}
\langle \De_{\eta_0}^{1/4}x\eta_0,\De_{\eta_0}^{1/4}b\eta\rangle
&=
\langle \De_{\eta_0}^{1/2}x\eta_0,y\eta_0\rangle
=
\langle Jy'\eta_0, J\De_{\eta_0}^{1/2}x\eta\rangle
\\
&=
\langle JyJ\eta_0, x\eta_0\rangle
=
\langle xy'\eta_0,\eta_0\rangle.
\end{align*}
Since $xy'$ is positive
and $xy'=x^{1/2}y'x^{1/2}\leq\|y'\|x=\|y\|x$,
we are done.
\epf

By applying the above lemmas, we can make a c.p.\ operator from a c.p.\ map.

\bprop\label{makecp}
Let $\stand$ be a standard form of a $\sigma$-finite von Neumann algebra $M$ 
with cyclic and separating vectors $\xi_0, \eta_0\in P$.
Let $\Phi$ be a c.p.\ map on $M$ 
such that $\omega_{\eta_0}\circ\Phi\leq C\omega_{\xi_0}$ 
for some $C>0$. 
Then there exists a c.p.\ operator $T$ on $H$ 
with $\|T\|\leq (C\|\Phi\|)^{1/2}$ 
such that 
\[
T(\Delta_{\xi_0}^{1/4}x\xi_0)
=\Delta_{\eta_0}^{1/4}\Phi(x)\eta_0
\quad \text{for}\ x\in M.
\]
\eprop

\bpf
We define a positive sesquilinear $s_\Phi$ on $M\times M$ by
\[
s_\Phi(x, y)
:=\langle \Delta_{\eta_0}^{1/4}\Phi(x)\eta_0, 
\Delta_{\eta_0}^{1/4}\Phi(y)\eta_0\rangle
\quad \text{for}\ x, y\in M.
\]
Note that the corresponding modular operators $\Delta_{\xi_0}$ and $\Delta_{\eta_0}$ may not coincide. 
However, by \cite[Lemma 2.9]{haa1},  
we have $P=P_{\xi_0}=P_{\eta_0}$ and $J=J_{\xi_0}=J_{\eta_0}$
because $\xi_0, \eta_0\in P$. 
Then one can easily check that
\[
s_\Phi(x, y)\geq 0
\quad \text{for}\ x, y\in M^+.
\]
Moreover for $x\in M^+$, by Lemma \ref{lem:etaab}, we have 
\begin{align*}
s_\Phi(x, 1)
&=\langle \Delta_{\eta_0}^{1/4}\Phi(x)\eta_0, 
\Delta_{\eta_0}^{1/4}\Phi(1)\eta_0\rangle \\
&\leq\|\Phi(1)\|\langle \Phi(x)\eta_0, \eta_0\rangle \\
&\leq C\|\Phi\|\omega_{\xi_0}(x).
\end{align*}
By Lemma \ref{wo}, we obtain
\[
s_\Phi(x, x)
=\|\Delta_{\eta_0}^{1/4}\Phi(x)\eta_0\|^2
\leq C\|\Phi\|\|\Delta_{\xi_0}^{1/4}x\xi_0\|^2
\quad \text{for}\ x\in M.
\]
Hence there exists a bounded operator $T$ on $H$ 
with $\|T\|\leq (C\|\Phi\|)^{1/2}$, 
which is defined by 
\[
T(\Delta_{\xi_0}^{1/4}x\xi_0)
=\Delta_{\eta_0}^{1/4}\Phi(x)\eta_0
\quad \text{for}\ x\in M.
\]
Finally we show that $T$ is a c.p.\ operator. 
Let $(e_{i, j})$ be a system of matrix units for $\mat$. 
For $[x_{i, j}]\in \mat(M)^+$, we have
\begin{align*}
(T\otimes\mathrm{id}_n)(\Delta_{\xi_0}^{1/4}\otimes\mathrm{id}_n)
(\sum_{i, j=1}^nx_{i, j}\otimes e_{i, j})(\xi_0\otimes 1_n)
&=\sum_{i, j=1}^n T(\Delta_{\xi_0}^{1/4}x_{i, j}\xi_0)\otimes e_{i, j} \\
&=\sum_{i, j=1}^n \Delta_{\eta_0}^{1/4}\Phi(x_{i, j})\eta_0\otimes e_{i, j}.
\end{align*}
Since $\Phi$ is a c.p.\ map, $[\Phi(x_{i, j})]\in \mat(M)^+$. 
Hence $T$ is a c.p.\ operator. 
\epf

In Lemma \ref{lem:unif-bdd} and Theorem \ref{sigma},
we deal with possibly non-contractive c.p.\ operators.
So, we use the symbol $S$ 
for a \emph{not necessarily contractive}
c.p.\ operator.
Similarly, we employ the symbol $\Psi$ 
for a \emph{not necessarily contractive}
c.p.\ map.

\begin{lemma}
\label{lem:unif-bdd}
Let $M$ be a $\sigma$-finite von Neumann algebra with a faithful state
$\varphi\in M_*^+$.
Suppose that
there exists a net of compact c.p.\ operators
$S_n$ on $H_\vph$ such that
$S_n\to 1_{H_\vph}$ in the strong topology
and
$\sup_n\|S_n\|<\infty$.
Then there exists a net of normal c.c.p.\ maps
$\widetilde{\Phi}_m$ on $M$
and compact c.p.\ operators $\wdt{S}_m$ on $H_\vph$
with a new directed set
such that
\begin{itemize}
\item
$\widetilde{\Phi}_m\to\id_M$ in the point-ultraweak topology;

\item
$\sup_n\|\wdt{S}_m\|<\infty$;

\item
$\wdt{S}_m (\De_\vph^{1/4}x\xi_\vph)=\De_\vph^{1/4}\widetilde{\Phi}_m(x)\xi_\vph$
for $x\in M$.
\end{itemize}
\end{lemma}
\begin{proof}
Let $S_n$ be as stated above.
By Lemma \ref{cyclic},
we may and do assume that
$\xi_n:=S_n\xi_\varphi$ is cyclic and separating
by taking sufficiently large $n$ so that $\|S_n\|$
is uniformly bounded.
Let $\Theta_{\xi_\varphi}$ and $\Theta_{\xi_n}$ be the maps
given in Lemma \ref{co}.
Let $x\in M_{\rm sa}$.
Take $c>0$ so that $-c1\leq x\leq c1$.
Then $-c\xi_\vph\leq \De_\vph^{1/4}x\xi_\vph\leq c\xi_\vph$.
Applying $S_n$ to this inequality,
we obtain $-c\xi_n\leq S_n \De_\vph^{1/4}x\xi_\vph\leq c\xi_n$,
because $S_n$ is positive.
Employing Lemma \ref{co},
the operator $\Th_{\xi_n}^{-1}(S_n\De_\vph^{1/4}x\xi_\vph)$ in $M$
is well-defined.
Hence we can define a linear map $\Phi_n\colon M\to M$ by
\[
\Phi_n
=\Theta_{\xi_n}^{-1}\circ S_n\circ\Theta_{\xi_\varphi}.
\]
In other words,
\[
S_n(\Delta_\varphi^{1/4}x\xi_\varphi)
=\Delta_{\xi_n}^{1/4}\Phi_n(x)\xi_n
\quad \text{for}\ x\in M.
\]
It is easy to check that $\Phi_n$ is a normal unital
completely positive (u.c.p.) map.

\vspace{5pt}
\noindent
{\bf Step 1.}
We will show that
$\Phi_n\to\mathrm{id}_ M$ in the point-ultraweak topology.
\vspace{5pt}

Since normal functionals of the form $\omega_{y'\xi_\varphi}$ with $y'\in M'$ 
span a dense subspace in $M_*$, 
it suffices to show that  
\begin{equation}\label{claim0}
\langle \Phi_n(x)\xi_\varphi, y'\xi_\varphi\rangle
\to \langle x\xi_\varphi, y'\xi_\varphi\rangle
\quad \text{for}\ x\in M,\ y'\in M'_{\mathrm{sa}}.
\end{equation}
To prove it, we first claim that 
\begin{equation}\label{claim1}
\|\Delta_{\xi_n}^{1/4}\Phi_n(x)\xi_n-\Delta_\varphi^{1/4}x\xi_\varphi\|
\to 0.
\end{equation}
Indeed, since $S_n\to 1_{H_\varphi}$ in the strong topology, 
we have 
\[
\|S_n(\Delta_\varphi^{1/4}x\xi_\varphi)-\Delta_\varphi^{1/4}x\xi_\varphi\|
\to 0.
\] 
Hence our claim (\ref{claim1}) follows. 
Secondly we claim that 
\begin{equation}\label{claim2}
\|\Delta_{\xi_n}^{-1/4}y'\xi_n-\Delta_\varphi^{-1/4}y'\xi_\varphi\|
\to 0.
\end{equation}
Indeed, if we set $y:=J y'J\in M_{\mathrm{sa}}$, 
then it is equivalent to the condition
\[
\|\Delta_{\xi_n}^{1/4}y\xi_n-\Delta_\varphi^{1/4}y\xi_\varphi\|
\to 0.
\]
Since 
\[
\Delta_\varphi^{1/4}y\xi_\varphi
=(J+1)(\Delta_\varphi^{1/4}+\Delta_\varphi^{-1/4})^{-1}y\xi_\varphi
\]
and 
\[
\Delta_{\xi_n}^{1/4}y\xi_n
=(J+1)(\Delta_{\xi_n}^{1/4}+\Delta_{\xi_n}^{-1/4})^{-1}y\xi_n,
\]
our claim (\ref{claim2}) is also equivalent to the condition 
\[
\|(\Delta_{\xi_n}^{1/4}+\Delta_{\xi_n}^{-1/4})^{-1}y\xi_n-(\Delta_\varphi^{1/4}+\Delta_\varphi^{-1/4})^{-1}y\xi_\varphi\|
\to 0.
\]
However it easily follows from Lemma \ref{ara} 
and $\|\xi_n-\xi_\varphi\|\to 0$. 
Thus to prove (\ref{claim0}), 
it suffices to show that 
\[
\langle \Phi_n(x)\xi_n, y'\xi_n\rangle\to\langle x\xi_\varphi, y'\xi_\varphi\rangle,
\]
because $\|\xi_n-\xi_\varphi\|\to 0$. 
By (\ref{claim2}), 
there is a constant $C_{y'}>0$ and $n_0$
such that 
\[
\|\Delta_{\xi_n}^{-1/4}y'\xi_n\|
\leq C_{y'}
\quad \text{for all}\ n\geq n_0.
\] 
By using (\ref{claim1}) and (\ref{claim2}), we have 
\begin{align*}
&|\langle\Phi_n(x)\xi_n, y'\xi_n\rangle-\langle x\xi_\varphi, 
y'\xi_\varphi\rangle| \\
&\hspace{3mm}=|\langle\Delta_{\xi_n}^{1/4}\Phi_n(x)\xi_n, 
\Delta_{\xi_n}^{-1/4}y'\xi_n\rangle
-\langle \Delta_{\varphi}^{1/4}x\xi_\varphi, 
\Delta_{\varphi}^{-1/4}y'\xi_\varphi\rangle| \\
&\hspace{3mm}\leq|\langle\Delta_{\xi_n}^{1/4}\Phi_n(x)\xi_n
-\Delta_\varphi^{1/4}x\xi_\varphi, 
\Delta_{\xi_n}^{-1/4}y'\xi_n\rangle|
+|\langle \Delta_{\varphi}^{1/4}x\xi_\varphi, 
\Delta_{\xi_n}^{-1/4}y'\xi_n
-\Delta_{\varphi}^{-1/4}y'\xi_\varphi\rangle| \\
&\hspace{3mm}\leq C_{y'}\|\Delta_{\xi_n}^{1/4}\Phi_n(x)\xi_n
-\Delta_\varphi^{1/4}x\xi_\varphi\|
+\|\Delta_{\varphi}^{1/4}x\xi_\varphi\|
\|\Delta_{\xi_n}^{-1/4}y'\xi_n-\Delta_{\varphi}^{-1/4}y'\xi_\varphi\| \\
&\hspace{3mm}\to 0.
\end{align*}
Therefore we obtain our claim (\ref{claim0}), 
that is, $\Phi_n\to\mathrm{id}_ M$ in the point-ultraweak topology.

\vspace{5pt}
\noindent
{\bf Step 2.}
We will make a small perturbation of $\Phi_n$.
\vspace{5pt}

Put $\varphi_n:=\omega_{\xi_n}\in M_*^+$. 
Since $\|\xi_n-\xi_\varphi\|\to 0$, 
we have $\|\varphi_n-\varphi\|\to 0$ 
by the Araki--Powers--St\o rmer inequality. 
If we set $\psi_n:=\varphi+(\varphi-\varphi_n)_-$, then $\varphi_n\leq\psi_n$. 
Thanks to Sakai's Radon--Nikodym theorem \cite[Theorem 1.24.3]{sak},
there exists $h_n\in M$ with $0\leq h_n\leq 1$ 
such that $\varphi_n(x)=\psi_n(h_n xh_n)$ for $x\in M$. 
We define a c.p.\ map $\Psi_n\colon M\to M$ by
\[
\Psi_n(x)
:=h_n xh_n+(\varphi-\varphi_n)_-(h_n xh_n)1
\quad \text{for}\ x\in M.
\]
Note that $\|\psi_n-\varphi\|=\|(\varphi-\varphi_n)_-\|\leq\|\varphi-\varphi_n\|\to 0$. 
Since 
\begin{align*}
\varphi(1-h_n^2)
&\leq\psi_n(1-h_n^2) \\
&=\psi_n(1)-\varphi_n(1) \\
&=\|\psi_n-\varphi_n\| \\
&\leq\|\psi_n-\varphi\|+\|\varphi_n-\varphi\|\to 0,
\end{align*}
we have $(1-h_n^2)^{1/2}\to 0$ in the strong topology. 
Moreover since 
\[
\|(1-h_n)\xi\|^2
=\langle(1-h_n)^2\xi, \xi\rangle\leq\langle(1-h_n^2)\xi, \xi\rangle
=\|(1-h_n^2)^{1/2}\xi\|
\quad \text{for}\ \xi\in H_\varphi,
\]
we have $h_n\to 1$ in the strong topology. 
Consequently, for $x\in M$, 
we have $h_n xh_n\to x$ in the strong topology. 
Therefore $\Psi_n\to \mathrm{id}_ M$ in the point-ultraweak topology.
Since
\[
\Psi_n(1)
=h_n^2+(\varphi-\varphi_n)_-(h_n^2)1
\leq 1+\|\varphi-\varphi_n\|
=:C_n\to 1,
\] 
a c.p.\ map $\Phi'_n:=\Psi_n/C_n$ is contractive 
such that $\Phi'_n\to\mathrm{id}_ M$ in the point-ultraweak topology. 

Moreover for $x\in M^+$ we have 
\begin{align*}
\varphi\circ\Phi'_n(x)
&=\frac{1}{C_n}\varphi(\Psi_n(x))
=\frac{1}{C_n}\psi_n(h_n xh_n)\\
&=\frac{1}{C_n}\varphi_n(x)
\leq\varphi_n(x).
\end{align*}
By Proposition \ref{makecp}, 
there exists a c.c.p.\ operator $T'_n$ on $H_\varphi$ by
\[
T'_n(\Delta_n^{1/4}x\xi_n)
:=\Delta_\varphi^{1/4}\Phi'_n(x)\xi_\varphi
\quad \text{for}\ x\in M.
\]
Since $\Phi'_n\to\mathrm{id}_ M$ in the point-ultraweak topology,
$T'_n\to 1_{H_\varphi}$ in the weak topology.

Now we define a normal c.c.p.\ map $\widetilde{\Phi}_n:=\Phi'_n\circ\Phi_n$ on $M$ 
and a c.p.\ compact operator $\widetilde{S}_n:=T'_n S_n$ on $H_\varphi$
which satisfies $\sup_n\|\widetilde{S}_n\|<\infty$.
Then we have
\[
\widetilde{S}_n(\Delta_\varphi^{1/4}x\xi_\varphi)
=\Delta_\varphi^{1/4}\widetilde{\Phi}_n(x)\xi_\varphi
\quad \text{for}\ x\in M.
\]

We first claim that $\widetilde{S}_n\to 1_{H_\varphi}$
in the weak topology.
Indeed, for $\xi, \eta\in H_\varphi$, we have
\begin{align*}
|\langle \widetilde{S}_n\xi, \eta\rangle-\langle\xi, \eta\rangle|
&=|\langle T'_n S_n\xi, \eta\rangle-\langle\xi, \eta\rangle| \\
&\leq |\langle T'_n S_n\xi-T'_n\xi, \eta\rangle|
+|\langle T'_n\xi-\xi, \eta\rangle| \\
&\leq\|S_n\xi-\xi\| \|\eta\|+|\langle T'_n\xi-\xi, \eta\rangle| \\
&\to 0.
\end{align*}

Next we claim that $\widetilde{\Phi}_n\to\mathrm{id}_ M$ in the point-ultraweak topology. 
It suffices to show that 
\[
\langle\widetilde{\Phi}_n(x)\xi_\varphi, y'\xi_\varphi\rangle
\to
\langle x\xi_\varphi, y'\xi_\varphi\rangle,
\quad \text{for}\ x\in M, y'\in M'.
\]
Indeed, 
\begin{align*}
\langle\widetilde{\Phi}_n(x)\xi_\varphi, y'\xi_\varphi\rangle
&=\langle\Delta_\varphi^{1/4}\widetilde{\Phi}_n(x)\xi_\varphi,
\Delta_\varphi^{-1/4}y'\xi_\varphi\rangle \\
&=\langle T_n(\Delta_\varphi^{1/4}x\xi_\varphi), \Delta_\varphi^{-1/4}y'\xi_\varphi\rangle \\
&\to\langle\Delta_\varphi^{1/4}x\xi_\varphi, \Delta_\varphi^{-1/4}y'\xi_\varphi\rangle=\langle x\xi_\varphi, y'\xi_\varphi\rangle.
\end{align*}

Finally, by taking suitable convex combinations, 
we can arrange $\wdt{S}_n$ and $\widetilde{\Phi}_n$
so that we obtain c.p.\ compact operators $\wdt{S}_m$ on $H_\varphi$
and normal c.c.p.\ maps $\widetilde{\Phi}_m$ on $M$
so that $\wdt{S}_m\to 1_{H_\varphi}$ in the strong topology,
$\sup_m\|\wdt{S}_m\|<\infty$,
$\widetilde{\Phi}_m\to\mathrm{id}_ M$ in the point-ultraweak topology
and
\[
\wdt{S}_m(\Delta_\varphi^{1/4}x\xi_\varphi)
=\Delta_\varphi^{1/4}\widetilde{\Phi}_m(x)\xi_\varphi
\quad \text{for}\ x\in M.
\]
\end{proof}

Now we are ready to prove the main theorem in this section. 

\bthm\label{sigma}
Let $M$ be a $\sigma$-finite von Neumann algebra with a faithful state
$\varphi\in M_*^+$.
Then the following statements are equivalent:
\begin{itemize}
\item[(1)]
$M$ has the HAP;

\item[(2)]
There exists a net of compact c.p.\ operators
$S_n$ on $H_\vph$ such that
$S_n\to 1_{H_\vph}$ in the strong topology
and $\sup_n\|S_n\|<\infty$;

\item[(3)]
There exists
a net of normal c.c.p.\ maps $\Phi_n$ on $M$
satisfying the following conditions:

\begin{itemize}
\item[(i)]
$\vph\circ\Ph_n\leq\vph$ for all $n$;

\item[(ii)]
$\Phi_n\to\mathrm{id}_{M}$ in the point-ultraweak topology;

\item[(iii)]
The associated c.c.p.\ operators $T_n$ on $H_\vph$
defined below are compact
and
$T_n\to\mathrm1_{H_\varphi}$ in the strong topology:
\[
T_n(\Delta_\varphi^{1/4}x\xi_\varphi)
=\Delta_\varphi^{1/4}\Phi_n(x)\xi_\varphi
\quad \text{for}\ x\in M.
\]
\end{itemize}
\end{itemize}
\ethm

\bpf
The implications
(1)$\Rightarrow$(2)
and
(3)$\Rightarrow$(1) are trivial.

(2)$\Rightarrow$(3).
Let us take $\widetilde{\Ph}_m$ and $\widetilde{S}_n$ 
as in the previous lemma.
We will arrange normal c.c.p.\ maps $\widetilde{\Phi}_m$
so that $\varphi\circ\widetilde{\Phi}_m\leq\varphi$.

We define $\chi_m\colon M_*\to M_*$ by
$\chi_m(\om):=\om\circ\widetilde{\Ph}_m$
for $\om\in M_*$.
By the convexity argument,
we may assume that $\|\chi_m(\omega)-\omega\|\to 0$ 
for $\omega\in M_*$. 
Set $\varphi_m:=\chi_m(\varphi)$. 
Note that $\|\varphi_m-\varphi\|\to 0$. 
Since $\widetilde{\Phi}_m(1)\to 1$ in the ultraweak topology, 
we may also assume that $\varphi_m(1)\ne 0$. 
Since 
\[
\psi_m
:=\varphi_m+(\varphi_m-\varphi)_-\geq\varphi,
\] 
by Sakai's Radon--Nikodym theorem, 
there is $h_m\in M$ with $0\leq h_m\leq 1$ 
such that $\varphi(x)=\psi_m(h_m xh_m)$ 
for $x\in M$. 
Then we define a normal c.p.\ map $\Psi_m$ on $M$ by
\[
\Psi_m(x)
:=h_m xh_m+\frac{1}{\varphi_m(1)}(\varphi_m-\varphi)_-(h_m xh_m)1
\quad \text{for}\ x\in M.
\]
Note that 
\begin{align*}
\varphi_m\circ\Psi_m(x)
&=\varphi_m(h_m xh_m)+\frac{1}{\varphi_m(1)}(\varphi_m-\varphi)_-(h_m xh_m)\varphi_m(1) \\
&=\psi_n(h_n xh_n)=\varphi(x).
\end{align*}
Since 
\begin{align*}
\varphi(1-h_m^2)
&\leq\psi_m(1-h_m^2) \\
&=\psi_m(1)-\varphi(1) \\
&=\|\psi_m-\varphi\| \\
&\leq\|\varphi_m-\varphi\|+\|(\varphi_m-\varphi)_-\| \\
&\to 0,
\end{align*}
we have $h_m\to 1$ in the strong topology. 
Hence $h_m xh_m\to x$ in the strong topology for $x\in M$. 
Moreover, since 
\[
\|(\varphi_m-\varphi)_-(h_m xh_m)\|
\leq\|\varphi_m-\varphi\|\|x\|\to 0
\quad \text{for}\ x\in M,
\] 
we have $\Psi_m\to\mathrm{id}_ M$ in the point-ultraweak topology.

Note that 
\begin{align*}
\Psi_m(1)
&=h_m^2+\frac{1}{\varphi_m(1)}(\varphi_m-\varphi)_-(h_m^2) \\
&\leq 1+\frac{1}{\varphi_m(1)}\psi_m(h_m^2) \\
&= 1+\frac{1}{\varphi_m(1)}=:C_m\to 1,
\end{align*}
and for $x\in M^+$, 
\begin{align*}
\varphi\circ\Psi_m(x)
&=\varphi(h_m xh_m)+\frac{1}{\varphi_m(1)}(\varphi_m-\varphi)_-(h_m xh_m) \\
&\leq C_m\psi_m(h_m xh_m) \\
&=C_m\varphi(x). 
\end{align*}
By Proposition \ref{makecp}, 
we obtain a c.p.\ operator $S_m$ on $H_\varphi$ 
with $\|S_m\|\leq C_m$ 
such that 
\[
S_m(\Delta_\varphi^{1/4}x\xi_\varphi)
=\Delta_\varphi^{1/4}\Psi_m(x)
\xi_\varphi\quad \text{for}\ x\in M.
\]
We may and do assume that $\sup_m\|S_m\|\leq\sup_mC_m<\infty$.
Notice that $S_m\to 1_{H_\varphi}$ in the weak topology, 
because $\Psi_m\to\mathrm{id}_M$ in the point-ultraweak topology. 

Finally we define a normal c.p.\ map $\Psi'_m:=\widetilde{\Phi}_m\circ\Psi_m$ on $M$ 
and a c.p.\ compact operator $S'_m:=\widetilde{S}_m S_m$ on $H_\varphi$.  
Then $\varphi\circ\Psi'_m=\varphi$ 
and
\[
S'_m(\Delta_\varphi^{1/4}x\xi_\varphi)
=\Delta_\varphi^{1/4}\Psi'_m(x)\xi_\varphi
\quad \text{for}\ x\in M.
\]
Moreover for $\omega\in M_*$, 
we have
\begin{align*}
|\langle \Psi'_m(x)-x, \omega\rangle|
&\leq|\langle \Psi_m(x), \chi_m(\omega)-\omega\rangle|+|\langle \Psi_m(x)-x, \omega\rangle| \\
&\leq C_m\|x\|\|\chi_m(\omega)-\omega\|+|\langle \Psi_m(x)-x, \omega\rangle| \\
&\to 0.
\end{align*}
Therefore $\Psi'_m\to\mathrm{id}_ M$ in the point-ultraweak topology, 
and thus $S'_m\to 1_{H_\varphi}$ in the weak topology,
because $\sup_m\|S'_m\|<\infty$. 

Note that 
\begin{align*}
\Psi'_m(1)
&=\Phi_m(h_m^2)+\frac{1}{\varphi_m(1)}(\varphi_m-\varphi)_-(h_m^2)\Phi_m(1) \\
&\leq1+\frac{\|\varphi_m-\varphi\|}{\varphi_m(1)}
=:C'_m\to 1.
\end{align*}
We define a normal c.c.p.\ map $\Phi_m$ on $M$ by 
$\Phi_m
:=\Psi'_m/C'_m$.

Note that $\varphi\circ\Phi_m\leq\varphi$ 
and $\Phi_m\to\mathrm{id}_ M$ in the point-ultraweak topology. 
By Proposition \ref{makecp}, 
we have a c.c.p. operator $T_m$ on $H_\varphi$, 
which is given by
\[
T_m(\Delta_\varphi^{1/4} x\xi_\varphi)
=\Delta_\varphi^{1/4}\Phi_m(x)\xi_\varphi
\quad \text{for}\ x\in M.
\]
Then
$T_m=S'_m/C'_m$ is compact
and $T_m\to 1_{H_\varphi}$ in the weak topology.
By the convexity argument,
we may and do assume that
$\Phi_n\to\mathrm{id}_ M$ in the point-ultraweak topology,
$T_n\to 1_{H_\varphi}$ in the strong topology
and, moreover,
$\varphi\circ\Phi_n\leq\varphi$
and 
\[
T_n(\Delta_\varphi^{1/4}x\xi_\varphi)
=\Delta_\varphi^{1/4}\Phi_n(x)\xi_\varphi
\quad \text{for}\ x\in M.
\]
\epf

\brem\label{torperem}
The proof of Theorem \ref{sigma} is essentially based on the one of \cite{tor}. 
The proof above can be also applied to show Theorem \ref{tor}. 
Also note that
we have proved the existence of c.c.p.\ maps $\Ph_n$
such that
$\vph\circ\Ph_n=\la_n\vph$
for some $0<\la_n\leq1$.
In particular, $\Ph_n$ is faithful.
\erem


\subsection{Commutativity of c.c.p.\ operators with modular groups}


In this subsection,
we study the Haagerup approximation property
such that a net of c.c.p.\ compact operators
are commuting a modular group.

\bdf
Let $M$ be a von Neumann algebra
with a f.n.s.\ weight $\vph$.
We will say that
$M$ has the $\vph$-\emph{Haagerup
approximation property}
($\vph$-HAP)
if
c.c.p.\ compact operators
$T_n$ introduced in Definition \ref{defn:HAP}
are moreover commuting with
$\De_\vph^{it}$ for all $t\in\R$.
\edf

In this case,
we can take unital $\varphi$-preserving $\Ph_n$'s as shown below.

\bthm\label{thm:sigma2}
Let $M$ be a $\sigma$-finite von Neumann algebra
with a faithful state $\varphi\in M_*^+$. 
If $M$ has the $\vph$-HAP,
then
there exist a net of c.c.p.\ compact operators $T_n$ on $H_\varphi$
with $T_n\to\mathrm1_{H_\varphi}$ in the strong topology,
and a net of normal u.c.p.\ maps $\Phi_n$ on $M$
with $\Phi_n\to\mathrm{id}_{M}$ in the point-ultraweak topology
such that
\benu
\item
$\vph\circ\Ph_n=\vph$
for all $n$.

\item
$T_n(\Delta_\varphi^{1/4}x\xi_\varphi)
=\Delta_\varphi^{1/4}\Phi_n(x)\xi_\varphi
\quad \text{for}\ x\in M$
for all $n$;
\eenu
\ethm

\bpf
Suppose that $M$ has the $\vph$-HAP.
Recall the proof of Theorem \ref{sigma}.
We let the starting $T_n$ be commuting
with $\De_\vph^{it}$
for all $t\in\R$.
Then it is not so difficult to
check that the last $\Ph_n$ is commuting
with $\si_t^\vph$ for all $t\in\R$.
So, we have $T_n$ and $\Ph_n$
stated in Theorem \ref{sigma}
and they are commuting with the modular group.
Thus we have c.c.p.\ compact operators $T_n$
on $H_\varphi$
and normal c.c.p.\ maps $\Phi_n$ on $M$ 
such that
\begin{itemize}
\item
$
T_n(\Delta_\varphi^{1/4}x\xi_\varphi)
=\Delta_\varphi^{1/4}\Phi_n(x)\xi_\varphi
\quad \text{for all}\ x\in M$;
%

\item
$\sigma_t^\varphi\circ\Phi_n
=\Phi_n\circ\sigma_t^\varphi
\quad \text{for all}\ t\in\R$.
\end{itemize}

We will make a small perturbation of $\Ph_n$
so that its perturbation is unital.
%
%
Set $\varphi_n:=\varphi\circ\Phi_n$.
Then $\varphi_n\circ\sigma_t^\varphi=\varphi_n$ for $t\in\R$. 
By \cite[Thereom 5.12]{pt}, 
there exists $h_n\in M_\varphi$ 
with $0\leq h_n\leq 1$ 
such that $\varphi_n(x)=\varphi(h_n x)$ for $x\in M$, 
where $M_\varphi$ denotes the centralizer of $\varphi$,
\[
M_\varphi
:=
\{x\in M\mid x\vph=\vph x\}
=
\{x\in M \mid \sigma_t^\varphi(x)=x
\quad \text{for}\ t\in\R\}.
\] 
Note that $\varphi_n(1)=\varphi(h_n)$. 
We may assume that $h_n\ne 1$. 
We set 
\[
x_n
:=\frac{1}{\varphi(1-h_n)}(1-\Phi_n(1))
\quad \text{and}\ 
y_n
:=1-h_n.
\]
Next we define a normal c.p.\ map $\Phi_n$ on $M$ by
\[
\Phi_n(x)
:=\Phi_n(x)+\varphi(y_n x)x_n
\quad \text{for}\ x\in M.
\]
Then $\varphi\circ\Phi_n=\varphi$. 
By Proposition \ref{makecp}, 
we obtain a c.p.\ operator $S_n$ on $ M_\varphi$ by
\[
S_n(\Delta_\varphi^{1/4}x\xi_\varphi)
:=\Delta_\varphi^{1/4}\Phi_n(x)\xi_\varphi
\quad \text{for}\ x\in M.
\]
Note that $S_n$ is compact, because
\begin{align*}
S_n(\Delta_\varphi^{1/4}x\xi_\varphi)
&=\Delta_\varphi^{1/4}\Phi_n(x)\xi_\varphi \\
&=\Delta_\varphi^{1/4}\Phi_n(x)\xi_\varphi+\varphi(y_n x)\Delta_\varphi^{1/4}x_n\xi_\varphi \\
&=T_n(\Delta_\varphi^{1/4}x\xi_\varphi)+\varphi(y_n x)\Delta_\varphi^{1/4}x_n\xi_\varphi, 
\end{align*}
Moreover 
\begin{align*}
\Phi_n(1)
&=\Phi_n(1)+\varphi(y_n)x_n \\
&=\Phi_n(1)+\varphi(1-h_n)\frac{1}{\varphi(1-h_n)}(1-\Phi_n(1)) \\
&=1.
\end{align*}
Finally since $y_n\in M_\varphi$, we have 
\begin{align*}
0&\leq\Psi_n(x)-\Phi_n(x)
= \varphi(y_n x )x_n
\\
&\leq \|x\|\varphi(y_n)x_n
= \|x\|(1-\Phi_n(1))
\quad \text{for}\ x\in M^+,
\end{align*}
Therefore $\Psi_n\to\mathrm{id}_ M$ in the point-ultraweak topology. 
%
\epf

\bthm
Let $(M_1, \varphi_1)$
and $(M_2, \varphi_2)$
be two $\sigma$-finite von Neumann algebras
with faithful normal states. 
If $M_i$ has the $\vph_i$-HAP, $i=1,2$,
then
the free product
$(M_1, \varphi_1)\star(M_2, \varphi_2)$
has the $\vph_1\star\vph_2$-HAP.
\ethm

\bpf
The proof is essentially given in \cite[Proposition 3.9]{boc}. 
We will give a sketch of a proof.
Assume that for $i=1,2$,
there exists a net of normal u.c.p.\ maps $\Phi_n^i$ on $ M$ 
such that $\varphi_i\circ\Phi_n^i=\varphi_i$ 
and $\Phi_n^i\to\mathrm{id}_{M_i}$ in the point-ultraweak topology. 
The corresponding c.c.p.\ compact operators $T_n^i$ on $H_{\varphi_i}$
are defined by  
\[
T_n^i(\Delta_{\varphi_i}^{1/4}x\xi_{\varphi_i})
=\Delta_{\varphi_i}^{1/4}\Phi_n^i(x)\xi_{\varphi_i}
\quad \text{for}\ x\in M_i.
\]
Set $(M, \varphi):=(M_1, \varphi_1)\star(M_2, \varphi_2)$. 
Then we obtain normal u.c.p.\ maps $\Phi_n:=\Phi_n^1\star\Phi_n^2$
such that
$\varphi\circ\Phi_n=\varphi$
and
$\Ph_n$ is commuting with $\si^\vph$.
We write $H_{\varphi_i}^\circ:=\ker\varphi_i$ for $i=1, 2$.
Since $T_n^i=1\oplus (T_n^i)^\circ$ on 
$H_{\varphi_i}
=\C\xi_{\varphi_i}\oplus H_{\varphi_i}^\circ$, 
we can define 
$T_n
:=T_n^1\star T_n^2$ 
on $(H, \xi):=(H_{\varphi_1}, \xi_{\varphi_1})\star(H_{\varphi_2}, \xi_{\varphi_2})$ by 
\begin{align*}
T_n\xi
&=\xi, \\
T_n(\xi_{i_1}\otimes\cdots\otimes\xi_{i_n})
&=(T_n^{i_1})^\circ\xi_{i_1}\otimes\cdots\otimes(T_n^{i_n})^\circ\xi_{i_n}\quad \text{for}\ i_1\ne\cdots\ne i_n.
\end{align*}
Then each $T_n$ is the corresponding c.c.p.\ compact operator with $\Phi_n$, 
and $T_n\to 1_H$ in the strong topology. 
\epf

\brem
Let $M$ be a $\sigma$-finite von Neumann algebra 
with a faithful state $\varphi\in M_*^+$. 
Suppose that $M$ has the HAP for $\varphi$ 
in the sense of \cite[Definition 6.3]{dfsw}, 
i.e., there exist a net of $\varphi$-preserving 
normal u.c.p.\ maps $\Phi_n$ on $M$ with $\Phi_n\to\mathrm{id}_M$ 
in the point-ultraweak topology, 
and a net of compact contractions $T_n$ on $H_\varphi$ 
with $T_n\to 1_{H_\varphi}$ in the strong topology 
such that
\[
T_n (x\xi_\varphi)
=\Phi_n(x)\xi_\varphi
\quad \text{for}\ x\in M.
\]
If the above normal u.c.p.\ normal maps $\Phi_n$
satisfy
\[
\sigma_t^\varphi\circ\Phi_n
=\Phi_n\circ\sigma_t^\varphi
\quad \text{for all}\ t\in\R,
\]
then $M$ has the $\vph$-HAP in our sense.
Indeed, for $x\in M^+$, as in \cite[VIII \S 2 Lemma 2.3]{t2}, we put
\[
x_\gamma
:=\sqrt{\frac{\gamma}{\pi}}
\int_{\R}
\exp(-\gamma t^2)\sigma_t^\varphi(x)\,dt.
\]
Then $x_\gamma$ is entire for $\gamma>0$.
Hence 
\begin{align*}
T_n (\Delta_\varphi^{1/4}x_\gamma\xi_\varphi)
&=T_n (\sigma_{-i/4}^\varphi(x_\gamma)\xi_\varphi)
=\Phi_n(\sigma_{-i/4}^\varphi(x_\gamma))\xi_\varphi \\
&=\sigma_{-i/4}^\varphi(\Phi_n(x_\gamma))\xi_\varphi
=\Delta_\varphi^{1/4}\Phi_n(x_\gamma)\xi_\varphi.
\end{align*}
Since $x_\gamma\to x$ in the ultraweak topology as $\gamma\to+\infty$, and 
\[
\Delta_\varphi^{1/4}(x_\gamma\xi_\varphi-x\xi_\varphi)
=(J_\varphi+1)(\Delta_\varphi^{1/4}+\Delta_\varphi^{-1/4})^{-1}(x_\gamma\xi_\varphi-x\xi_\varphi),
\]
we have
\[
T_n(\Delta_\varphi^{1/4}x\xi_\varphi)
=\Delta_\varphi^{1/4}\Phi_n(x)\xi_\varphi
\quad \text{for}\ x\in M.
\] 
Therefore the above compact contraction $T_n$ is,
in fact,
a c.p.\ operator on $H_\varphi$, 
and thus $M$ has the $\vph$-HAP.
\erem

The following result states that
the combination of the HAP and
the existence of an almost periodic
state $\vph$ implies $\vph$-HAP.

\bthm
\label{thm:vph-T}
Let $M$ be a $\sigma$-finite von Neumann algebra
with the HAP.
If there exists
a faithful almost periodic state $\vph\in M_*^+$,
then
$M$ has the $\vph$-HAP.
\ethm

\begin{proof}
Thanks to \cite{co2},
there exist a compact group $G$,
an action $\si\col G\to\Aut(M)$
and a continuous group homomorphism
$\rho\col\R\to G$
such that
$\si_t^\vph=\si_{\rho(t)}$ for $t\in\R$
and $\rho$ has the dense range.
Let $U\col G\to \bB(H_\vph)$ be the associated
unitary representation which implements
$\si$.
Note that $U_g P=P$ and
$J_\vph U_g=U_gJ_\vph$.
Hence, $U_g$ is a c.p.\ unitary operator.

Let $(T_n)$ be a net of c.c.p.\ compact operators
such that $T_n\to1_{H_\vph}$ in the strong topology.
We put
\[
\widetilde{T}_n:=\int_G U_g T_nU_g^*\,dg.
\]
Then $\widetilde{T}_n$ belongs to $\bK(H_\vph)$
because the compactness of $T$
implies the norm continuity of
the map
$G\ni g\mapsto U_g T_nU_g^*\in\bK(H_\vph)$.
It is clear that $\widetilde{T}_n$ is contractive
and commuting with $\De_\vph^{it}=U_{\rho(t)}$
for all $t\in\R$.
We will show the complete positivity
of $\widetilde{T}_n$.
Let $[\xi_{i,j}]\in P^{(m)}$, $m\in\N$.
Take $x_1,\dots,x_m\in M$.
Then we have
\begin{align*}
\sum_{i,j=1}^m
x_iJ_\vph x_jJ_\vph
\widetilde{T}_n\xi_{i,j}
&=
\int_G
dg
\sum_{i,j=1}^m
x_iJ_\vph x_jJ_\vph
U_g T_n U_g^*\xi_{i,j}
\\
&=
\int_G
dg\,
U_g
\sum_{i,j=1}^m
\si_{g^{-1}}(x_i)
J_\vph\si_{g^{-1}}(x_j)J_\vph
T_n U_g^*\xi_{i,j}.
\end{align*}
Since $T_n U_g^*$ is a c.p.\ operator,
$\si_{g^{-1}}(x_i)
J_\vph\si_{g^{-1}}(x_j)J_\vph
T_n U_g^*\xi_{i,j}\in P$
for each $g\in G$,
and the integration above
belongs to $P$.
Hence $\widetilde{T}_n$
is a c.p.\ operator.

We will check that $\widetilde{T}_n\to1_{H_\vph}$
in the strong topology.
Let $\xi\in H_\vph$.
Then the set $K:=\{U_g^*\xi\mid g\in G\}$
is norm compact,
and $T_n\to1_{H_\vph}$ uniformly on $K$
in the strong topology.
Thus we are done.
\end{proof}

\bcor
Let $M$ be a $\sigma$-finite von Neumann algebra
with the HAP.
If there exists
a faithful almost periodic state $\vph\in M_*^+$,
then
there exists a net of normal u.c.p.\ maps $\Ph_n$
on $M$ such that
\begin{enumerate}
\item[\rm{(1)}]
$\vph\circ\Ph_n=\vph$ for all $n$;

\item[\rm{(2)}]
$\Ph_n\circ\si_t^\vph
=\si_t^\vph\circ\Ph_n$
for all $t\in\R$;

\item[\rm{(3)}]
$\Ph_n\to\id_M$ in the point-ultraweak topology;

\item[\rm{(4)}]
The following
associated operator $T_n$ on $H_\vph$
is compact:
\[
T_n (x\xi_\vph)=\Ph_n(x)\xi_\vph
\quad\mbox{for }x\in M.
\]
\end{enumerate}
\ecor

\bex
The following examples have the HAP for $\varphi$
in the sense of \cite[Definition 6.3]{dfsw}. 
%
%
All known examples so far have the $\vph$-HAP.

\begin{itemize}
\item The free Araki--Woods factors \cite{hr};
\item The free quantum groups \cite{cfy};
\item The duals of quantum permutation groups \cite{br1};
\item The duals of Wang's quantum automorphism groups \cite{br2};
\item The duals of quantum reflection groups \cite{le}.
\end{itemize}
\eex

\brem
The Haar state $h$ on a compact quantum group $G$
is almost periodic.
Thus if $L^\infty(G)$, the function algebra on $G$,
has the HAP,
then $L^\infty(G)$ has the $h$-HAP.
\erem

  
\section{Crossed products}


Let $G$ be a locally compact group
and $\al$ an action of $G$ on a von Neumann algebra $M$.
Our main result in this section is the following.

\bthm\label{thm:cross}
If $M\rti_\al G$ has the HAP, then so does $M$.
\ethm

To prove this,
we may and do assume that $M$ is properly infinite
by studying the tensor product $\bB(\ell_2)\oti M$
and the action $\id\oti\al$.
Let $\be$ be the bidual action of $\al$
on $M\otimes \bB(L^2(G))$.
Then $\be$ has the invariant weight
and $\be$ is cocycle conjugate to $\al\otimes\id$.
Thus we may and do assume that
there exists a weight $\vph$ on $M$
such that $\vph\circ\al_t=\vph$
for all $t\in G$.

Let $N:=M\rti_\al G$ be the von Neumann algebra
generated by
the copy of $M$, $\pi_\al(M)$,
and the copy of $G$, $\la^\al(G)$
as defined below:
\[
(\pi_\al(x)\xi)(s)=\al_{s^{-1}}(x)\xi(s),
\quad
(\la^\al(t)\xi)(s)=\xi(t^{-1}s)
\]
for
$x\in M$, $s,t\in G$
and
$\xi\in H_\vph\oti L^2(G)$.

Let $\hvph$ be the dual weight of $\vph$.
Then for all $x\in n_\vph$
and $f\in C_c(G)$,
we obtain
\[
\hvph((\la^\al(f)\pi_\al(x))^*\la^\al(f)\pi_\al(x))
=
\vph(x^*x)\int_G |f(t)|^2\,dt.
\]
Hence $a:=\la^\al(f)\pi_\al(x)\in n_{\hvph}$ and
$\|\La_\hvph(a)\|=\|\La_\vph(x)\|_\vph\|f\|_2$.
Actually, it is known that
there exists a surjective isometry
from $H_{\hvph}$ onto $H_\vph\oti L^2(G)$
which maps
$\La_\hvph(a)$ to $\La_\vph(x)\oti f$.
Thus we will regard $H_\hvph=H_\vph\oti L^2(G)$
and
\[
\La_\hvph(\la^\al(f)\pi_\al(x))=\La_\vph(x)\oti f
\quad
\mbox{for }
x\in n_\vph,\ f\in L^2(G).
\]

Note that $\vph$ is $\al$-invariant,
and
$\la^\al(t)$ is fixed by $\si^\hvph$,
that is,
$\C\oti L(G)=\{\la^\al(t)\mid t\in G\}''$
is contained in the centralizer
$N_\hvph$.
The following formulae are frequently used:
\[
\si_t^\hvph(\pi_\al(x))
=
\pi_\al(\si_t^\vph(x)),
\quad
\si_t^\hvph(\la^\al(f))
=
\la^\al(f).
\]
for all
$t\in\R$, $x\in M$ and $f\in L^1(G)$.

Denote by $\De_G$ the modular function of $G$.
In the following, $dt$ denotes a left invariant Haar measure on $G$.
Then $L^1(G)$ is a Banach $*$-algebra
equipped
with the convolution product and the involution defined as follows:
\[
(f*g)(t):=\int_G f(s)g(s^{-1}t)\,ds,
\quad
f^*(t):=\De_G(t^{-1})\ovl{f(t^{-1})}
\]
for $f,g\in L^1(G)$ and $t\in G$.
We further recall the following useful formulae:
\[
d(st)=dt,\quad d(ts)=\De_G(s)dt,\quad d(t^{-1})=\De_G(t^{-1})dt.
\]


For $g\in C_c(G)$,
let us introduce the following map
$R_g\col H_\vph\to H_\hvph$
satisfying
\[
R_g\La_\vph(x):=\La_\hvph(\la^\al(g)\pi_\al(x)\la^\al(g)^*)
\quad
\mbox{for }
x\in n_\vph.
\]
This map is bounded since
\[
\La_\hvph(\la^\al(g)\pi_\al(x)\la^\al(g)^*)
=
J_\hvph\la^\al(g)J_\hvph
\La_\hvph(\la^\al(g)\pi_\al(x))
=
J_\hvph\la^\al(g)J_\hvph (\La_\vph(x)\oti g),
\]
and $\|R_g\|\leq \|g\|_1\|g\|_2$.
We will improve this estimate as follows.

\blem\label{lem:Rg}
Let $g\in C_c(G)$.
Then the following statements hold:
\benu
\item
$R_g$ is a c.p.\ operator;

\item
$\|R_g\|\leq\|\De_G^{-1/2}\cdot (g^* * g)\|_2$.
\eenu
\elem
\begin{proof}
(1).
Let $x\in m_\vph$ be an entire element with respect to $\si^\vph$.
Then $x J_\vph \La_\vph(x)=\La_\vph(x\si_{i/2}^\vph(x)^*)$,
and
\begin{align*}
R_gx J_\vph \La_\vph(x)
&=
R_g\La_\vph(x\si_{i/2}^\vph(x)^*)
\\
&=
\La_\hvph(\la^\al(g)\pi_\al(x\si_{i/2}^\vph(x)^*)\la^\al(g)^*)
\\
&=
\La_\hvph(\la^\al(g)\pi_\al(x)\cdot \si_{i/2}^\hvph(\la^\al(g)\pi_\al(x))^*)
\\
&=
\la^\al(g)\pi_\al(x)J_\hvph\La_\hvph(\la^\al(g)\pi_\al(x)),
\end{align*}
which belongs to $P_\hvph$.
Thus $R_gP_\vph\subs P_\hvph$.

Consider the action $\al\oti\id_n$ on $M\oti\mat$
for $n\geq1$.
Let $\widetilde{R}_g\col H_{\ps}\to H_{\hat{\ps}}$
be the map as defined above,
where $\ps:=\vph\oti\tr_n$.
We have proved that $\widetilde{R}_g$ is positive.
By the natural identification
$H_\ps=H_\vph\oti\mat$
and
$\hat{\ps}=\hvph\oti\tr_n$,
the map $\widetilde{R}_g=R_g\oti\id_n$
is positive.
Hence $R_g$ is $n$-positive for all $n$.

(2).
Let $x\in n_\vph$.
Then
\begin{align*}
\pi_\al(x)\la^\al(g)^*
&=
\pi_\al(x)\la^\al(g^*)
=
\pi_\al(x)\int_G g^*(t)\la^\al(t)\,dt
\\
&=
\int_G g^*(t)\la^\al(t)\pi_\al(\al_{t^{-1}}(x))\,dt.
\end{align*}
Since $\la^\al(g)\la^\al(t)=\De_G(t^{-1})\la^\al(g_{t^{-1}})$,
where $g_{t^{-1}}(s):=g(st^{-1})$,
we have
\[
\la^\al(g)\pi_\al(x)\la^\al(g)^*
=
\int_G \De_G(t^{-1})g^*(t)\la^\al(g_{t^{-1}})\pi_\al(\al_{t^{-1}}(x))
\,dt.
\]
Then
\begin{align*}
R_g\La_\vph(x)
&=
\int_G
\De_G(t^{-1})g^*(t)\,
\La_\vph(\al_{t^{-1}}(x))\oti g_{t^{-1}}
\,dt
\\
&=
\int_G
g^*(t^{-1})\,
\La_\vph(\al_{t}(x))\oti g_{t}
\,dt.
\end{align*}

Hence for $y\in n_\vph$,
we obtain
\begin{align*}
\langle R_g\La_\vph(x),R_g\La_\vph(y)\rangle
&=
\int_{G\times G}
g^*(t^{-1})\ovl{g^*(s^{-1})}\,
\langle
\La_\vph(\al_{t}(x))\oti g_{t},
\La_\vph(\al_{s}(y))\oti g_{s}
\rangle
\,dsdt
\\
&=
\int_{G\times G}
g^*(t^{-1})\ovl{g^*(s^{-1})}\,
\vph(y^*\al_{s^{-1}t}(x))
\langle g_{s^{-1}t},g\rangle
\,dsdt
\\
&=
\int_{G\times G}
g^*(t^{-1}s^{-1})\ovl{g^*(s^{-1})}\,
\vph(y^*\al_{t}(x))
\langle g_{t},g\rangle
\,dsdt.
\end{align*}
Since
\begin{align*}
\int_G
g^*(t^{-1}s^{-1})\ovl{g^*(s^{-1})}\,ds
&=
\int_G
g^*(t^{-1}s)\ovl{g^*(s)}
\De_G(s^{-1})\,ds
\\
&=
\int_G
\De_G(t^{-1})
\cdot
\De_G(t^{-1}s)^{-1}
g^*(t^{-1}s)\ovl{g^*(s)}
\,ds
\\
&=
\int_G
\De_G(t^{-1})
\cdot
\ovl{(g^*)^*(s^{-1}t)}\ovl{g^*(s)}
\,ds
\\
&=
\De_G(t^{-1})
\ovl{(g^**g)(t)},
\end{align*}
and
$\langle g_t,g\rangle
=(g^* *g)(t)$,
we have
\[
\langle R_g\La_\vph(x),R_g\La_\vph(y)\rangle
=
\int_G \De_G(t^{-1})|g^**g(t)|^2\vph(y^*\al_t(x))\,dt.
\]
This implies that
\begin{equation}
\label{eq:Rgstar}
R_g^*R_g\La_\vph(x)
=
\int_G \De_G(t^{-1})|g^**g(t)|^2\,\La_\vph(\al_t(x))\,dt,
\end{equation}
and
\[
\|R_g^*R_g\|
\leq
\int_G \De_G(t^{-1})|g^**g(t)|^2\,dt
=\|\De_G^{-1/2}\cdot (g^**g)\|_2^2.
\]
\end{proof}

\brem
If there exists a non-zero $x\in n_\vph\cap M^\al$,
then the equality (\ref{eq:Rgstar}) implies
$\|R_g\|=\|\De_G^{-1/2}\cdot (g^**g)\|_2$.
\erem

Now let $\cU$ be the collection
of all compact neighborhoods of the neutral element $e\in G$.
We will equip $\cU$ with the structure of the directed set
as $U\leq V$ if and only if $V\subs U$
for $U,V\in\cU$.

For each $U\in\cU$,
take a non-zero $g_U\in C_c(G)$ such that $\supp g_U\subs U$.
Now let
\[
k_U(t)
:=\|\De_G^{-1/2}\cdot (g_U^**g_U)\|_2^{-2}\De_G(t^{-1})|(g_U^**g_U)(t)|^2
\quad
\mbox{for }
t\in G.
\]
Note that $g_U^**g_U$ is non-zero since so is $g_U$.

The following lemma is a direct consequence of the definition.

\blem
\label{lem:k_U}
The function $k_U$ has the following properties:
\begin{itemize}
\item
$k_U(t)\geq 0$ for all $t\in G$;
\item
$\supp k_U\subs U^{-1}U$;
\item
$\int_G k_U(t)\,dt=1$.
\end{itemize}
In particular,
it follows for any continuous function $f$ on $G$ that
\[
\lim_{U}\int_G k_U(t)f(t)\,dt=f(e).
\]
\elem

\blem
\label{lem:S_U}
Let $R_{g_U}$ be as before.
Then the following statements hold:
\benu
\item
The operator
$S_U:=\|\De_G^{-1/2}\cdot (g_U^**g_U)\|_2^{-1}R_{g_U}$
is a c.c.p.\ operator from $H_\vph$ into $H_\hvph$;

\item
$S_U^*S_U\to1_{H_\varphi}$ in the strong topology of $\bB(H_\vph)$.
\eenu
\elem

\begin{proof}
(1).
It is clear from Lemma \ref{lem:Rg}
that $S_U$ is a c.c.p.\ operator. 

(2).
Let $x\in n_\vph$.
By (\ref{eq:Rgstar}),
we have
\[
\|S_U^*S_U\La_\vph(x)-\La_\vph(x)\|
\leq
\int_G k_U(t)\|\La_\vph(\al_t(x))-\La_\vph(x)\|\,dt.
\]
Applying Lemma \ref{lem:k_U} to
$f(t):=\|\La_\vph(\al_t(x))-\La_\vph(x)\|$,
we are done.
\end{proof}

Now we will present a proof of Theorem \ref{thm:cross}.

\begin{proof}[Proof of Theorem \ref{thm:cross}]
Let $\cF$ be the collection of all finite sets
contained in $n_\vph$.
It is trivial that $\{\La_\vph(x)\mid x\in F\}_{F\in\cF}$
forms a net of finite sets in $H_\vph$ such that
their union through $F\in\cF$ is dense in $H_\vph$.

Let $F\in\cF$ be a non-empty set.
Employing Lemma \ref{lem:S_U},
we can take $U_F\in\cU$ so that
\begin{equation}
\label{eq:SS}
\|S_{U_F}^*S_{U_F}\La_\vph(x)-\La_\vph(x)\|
<\frac{1}{|F|}
\quad \text{for}\ x\in F.
\end{equation}

Next, let $T_\gamma$ be a net of c.c.p.\ compact operators
on $H_\hvph$ such that $T_\gamma\to1$ in the strong topology of $\bB(H_\hvph)$.
Then we can find $\gamma_F$ such that
\begin{equation}
\label{eq:STS}
\|T_{\gamma_F}S_{U_F}\La_\vph(x)-S_{U_F}\La_\vph(x)\|
<\frac{1}{|F|}
\quad \text{for}\ x\in F.
\end{equation}

Now put $\widetilde{T}_F:=S_{U_F}^*T_{\gamma_F}S_{U_F}$.
Then $\widetilde{T}_F$ is a c.c.p. compact operator on $H_\vph$,
and by (\ref{eq:SS}) and (\ref{eq:STS}),
we have
\[
\|\widetilde{T}_F\La_\vph(x)-\La_\vph(x)\|
<\frac{2}{|F|}
\quad
\mbox{for all }
x,y\in F,\ F\in\cF.
\]
This implies that $\widetilde{T}_F\to1_{H_\varphi}$ in the strong topology. 
\end{proof}

\bcor\label{cor:abelian}
Let $G$ be a locally compact abelian group
and $\al$ an action on a von Neumann algebra.
Then
$M$ has the HAP
if and only if so does $M\rti_\al G$.
\ecor

\begin{proof}
The ``if'' part is nothing but Theorem \ref{thm:cross}.
Next we will prove the ``only if'' part.
Suppose that $M$ has the HAP.
Then so does $M\oti \bB(L^2(G))$ 
by Corollary \label{cor:tor} and Theorem \ref{tensor}.
The Takesaki duality states that
$M\oti \bB(L^2(G))$ is isomorphic to
$(M\rti_\al G)\rti_\hal \hat{G}$.
Hence $M\rti_\al G$ has the HAP by Theorem \ref{thm:cross}.
\end{proof}

It is well-known that the crossed product $M\rti_{\si^\vph}\R$
does not depend on the choice of an f.n.s.\ weight $\vph$.
So, we denote it by $\tM$ and call it the \emph{core} of $M$.
The reader is referred to \cite{ft}, \cite{t2} for the cores.

\bcor\label{cor:core}
Let $M$ be a von Neumann algebra
and $\tM$ the core.
Then
$M$ has the HAP if and only if so does $\tM$.
\ecor


\brem\label{rem:CS}
M. Caspers and A. Skalski independently
introduced the notion of the Haagerup approximation property
for arbitrary von Neumann algebras
in their setting.
One may wonder whether two definitions differ or not.
Actually, these formulations are equivalent as shown below
though we give an indirect proof using cores.
In either way,
a von Neumann algebra has the HAP if and only if so does its core.
(See \cite[Corollary 5.10, Theorem 6.6]{CS}.)
Thus we may and do assume that
$M$ is finite or of type II$_\infty$.
If $M$ is of type II$_\infty$,
then $M$ has a finite projection $e$ with central support $1$.
Considering the corner $eMe$,
we may and do assume that $M$ is finite.
(See \cite[Lemma 4.1, Proposition 5.9]{CS}.)
Then it is fairly trivial that our definition coincides
with \cite[Definition 3.1]{CS} for a faithful normal tracial state
by Theorem \ref{sigma}.
\erem


As an application of Corollary \ref{cor:core},
we will prove the following result which generalizes
Theorem \ref{thm:expectation}.

\bthm
\label{thm:norm1proj}
Let $N\subs M$ be an inclusion of von Neumann algebras.
Suppose that there exists a norm one projection from $M$ onto $N$.
If $M$ has the HAP, then so does $N$.
\ethm

To prove this,
we may assume that $N$ and $M$ are properly infinite
by considering $N\oti \bB(\ell^2)\subs M\oti \bB(\ell^2)$ if necessary.
Let $\tM$ be the core of $M$, which has the HAP by Corollary \ref{cor:core}.
Note that there exists a norm one projection from $\tM$ onto $M$
by averaging the dual action on $\tM$.
Thus we may assume that $M$ is semifinite.
Let $N=Q\rti_\th\R$ be a continuous decomposition of $N$
for some $\R$-action $\th$ on a semifinite von Neumann algebra $Q$.
By Corollary \ref{cor:abelian},
it suffices to prove that $Q$ has the HAP.

Therefore we may assume that $N$ and $M$ are semifinite.
Let $p\in N$ be a finite projection with central support 1 in $N$.
By Corollary \ref{lem:central1}, our task is reduced to prove that
$pNp$ has the HAP.
So, we may assume that $N$ is finite and also $\si$-finite
by usual reduction argument with Proposition \ref{increasing}.

In the following discussion,
$\ta_N$ and $\ta_M$ denote a faithful normal tracial state on $N$
and a f.n.s.\ tracial weight on $M$,
respectively.
Thanks to \cite[Theorem 5.1]{haa3},
there exists a unique f.n.s.\ operator valued weight
$T$ from $M$ onto $N$ such that $\ta_M=\ta_N\circ T$.

Recall the following lemma \cite[Lemma 3.7]{ana} originally due to Connes
(See \cite[p.102]{co2}).

\blem\label{lem:A-D}
Let $N$ and $M$ be as in Theorem \ref{thm:norm1proj}.
Then for any $\delta>0$ and a finite subset $F\subset N$,
there exists a normal state $\vph$ on $M$
such that
\begin{equation}
\label{eq:vphtaNleft}
\|\vph|_N-\ta_N\|_{N_*}<\delta,
\end{equation}
\begin{equation}
\label{eq:vphtaNright}
\|a\vph-\vph a\|_{M_*}<\delta
\quad
\mbox{for all }a\in F.
\end{equation}
\elem

In the following, we will use the notations
\[
|x|_{\ta_M}:=\ta_M(|x|),
\quad
\|x\|_{\ta_M}:=\ta_M(x^*x)^{1/2}
\quad
\mbox{for }x\in M.
\]
We prepare the notations $|\cdot|_{\ta_N}$
and $\|\cdot\|_{\ta_N}$ as well.
The important fact is that they satisfy
the triangle inequality by the tracial property.

\blem
\label{lem:norm1vep}
Let $N$ and $M$ be as in Theorem \ref{thm:norm1proj}.
Then for any $\vep>0$ and a finite subset $F\subset N$,
there exists $b\in n_{\ta_M}\cap M^+$ and a projection $e\in N$
such that
\begin{itemize}
\item
$\ta_M(b^2)\leq1$;

\item 
$(1-\vep)e\leq T(b^2)\leq (1+\vep)e$;

\item
$\ta_N(1-e)<\vep$;

\item
$\|\La_{\ta_M}(ab)-\La_{\ta_M}(ba)\|<\vep$
for all $a\in F$.
\end{itemize}
\elem

\begin{proof}
We may and do assume that $F$ consists of unitary operators.
Let us take $1\geq\de>0$ small enough so that
$10\de^{1/4}<\vep^2$,
$1-\vep<(1-\de^{1/4})^2$,
and $(1+\de^{1/4})^2<1+\vep$.
Applying Lemma \ref{lem:A-D} to $\de$ and $F$,
we obtain a state $\vph\in M_*$ satisfying (\ref{eq:vphtaNleft})
and (\ref{eq:vphtaNright}).
Take the unique vector $\xi\in P_M$ such that $\vph=\om_\xi$,
where $P_M$ denotes the natural cone of $M$
realized in the GNS Hilbert space $H_{\ta_M}$.
We may and do assume that $\xi=\La_{\ta_M}(b)$
for some positive $b\in n_{\ta_M}$.
Then we have
\[
\vph(x)=\ta_M(bxb)=\ta_M(b^2x)=\ta_N(T(b^2)x)
\quad
\mbox{for }x\in N.
\]
In particular, 
$1=\vph(1)=\ta_N(T(b^2))$,
and thus 
$h:=T(b^2)$ is an operator in $L^1(N,\ta_N)_+$,
where $L^1(N,\ta_N)_+$ denotes the positive operator
in $L^1(N,\ta_N)$, the non-commutative $L^1$-space
with respect to the finite von Neumann algebra $\{N,\ta_N\}$.
The $L^1$-norm is denoted by $|\cdot|_{\ta_N}$.
The $L^2$-space of $\{N,\ta_N\}$ and the $L^2$-norm
are denoted by $L^2(N,\ta_N)$ and $\|\cdot\|_{\ta_N}$
as well.
For more details about the non-commutative $L^p$-space 
with respect to a faithful normal semifinite tracial weight, 
the reader may refer to \cite[IX.2]{t2}. 
Then (\ref{eq:vphtaNleft}) 
implies
\begin{equation}
\label{eq:h-1}
|h-1|_{\ta_N}<\de.
\end{equation}
Applying the Araki--Powers--St\o rmer inequality to
(\ref{eq:vphtaNright}),
we have
\begin{equation}
\label{eq:ubuvph}
\|\La_{\ta_M}(ubu^*)-\La_{\ta_M}(b)\|
\leq
\|u\vph u^*-\vph\|^{1/2}
<\de^{1/2}.
\end{equation}
Thus our task is
to arrange the operator norm of $h$.
Using the Araki--Powers--St\o rmer inequality,
we have
\begin{equation}
\label{eq:h1}
\|h^{1/2}-1\|_{\ta_N}^2
\leq
|h-1|_{\ta_N}
<\de
\quad
\mbox{by }(\ref{eq:h-1}).
\end{equation}

Let $h=\int_0^{\infty}\la\,de(\la)$
be the spectral decomposition.
Set
\[
\al_\de:=(1-\de^{1/4})^2,
\quad
\be_\de:=(1+\de^{1/4})^2.
\]
We put 
\[
e_1:=e([0,\al_\de)),
\quad
e_2:=e((\be_\de,\infty]).
\]
Then it follows from (\ref{eq:h1}) that
\[
\de^{1/2}\ta_N(e_1)
\leq
\int_{[0,\al_\de)}
|\la^{1/2}-1|^2\,d\ta(e(\la))
\leq
\|h^{1/2}-1\|_{\ta_N}^2
<\de,
\]
and
\[
\de^{1/2}\ta_N(e_2)
\leq
\int_{(\be_\de,\infty]}
|\la^{1/2}-1|^2\,d\ta(e(\la))
\leq
\|h^{1/2}-1\|_{\ta_N}^2
<\de.
\]
Thus
\begin{equation}
\label{eq:spectrum}
|e_1|_{\ta_N}=\ta_N(e_1)<\de^{1/2},
\quad
|e_2|_{\ta_N}=\ta_N(e_2)<\de^{1/2}.
\end{equation}

Put $e:=e([\al_\de,\be_\de])\in N$
and $b':=(eb^2e)^{1/2}\in M$.
Then 
\[
\tau_N(1-e)=\tau_N(e_1)+\tau_N(e_2)<2\delta^{1/2}<\e
\]
and
\[
\tau_M(b'^2)=\ta_M(eb^2e)=\varphi(e)\leq 1.
\]
Moreover,
\[
T(b'^2)=T(eb^2e)=eT(b^2)e=ehe\leq \be_\de e
\leq (1+\vep)e,
\]
and, similarly, $(1-\vep)e\leq T(b'^2)$.

Next we have
\begin{align*}
|T(b'^2)-1|_{\ta_N}
&=
|ehe-1|_{\ta_N}
\leq
|e(h-1)e|_{\ta_N}
+
|e-1|_{\ta_N}
\\
&\leq
|h-1|_{\ta_N}
+
|e-1|_{\ta_N}
\\
&<
\de+2\de^{1/2}<3\de^{1/2}
\quad
\mbox{by }
(\ref{eq:h-1}), (\ref{eq:spectrum}).
\end{align*}
Let $(1-e)b^2=v|(1-e)b^2|$ be the polar decomposition
with a partial isometry $v$ in $M$.
Since
\begin{align*}
|(1-e)h|_{\ta_N}
\notag
&\leq
|(1-e)(h-1)|_{\ta_N}
+
|1-e|_{\ta_N}
\notag\\
&<
\de+2\de^{1/2}<3\de^{1/2},
\end{align*}
we have
\begin{align}
|b^2(1-e)|_{\ta_M}
&=
|(1-e)b^2|_{\ta_M}
\notag\\
&=
\ta_M(v^*(1-e)b^2)
=
\ta_M(bv^*(1-e)b)
\notag\\
&\leq
\ta_M(bv^*vb)^{1/2}\ta_M(b(1-e)b)^{1/2}
\notag\\
&\leq
\ta_M(b^2)^{1/2}\ta_M((1-e)b^2)^{1/2}
\notag\\
&=
\ta_N((1-e)h)^{1/2}
\notag\\
&=|(1-e)h|_{\ta_N}^{1/2}
\notag\\
&<\sqrt{3}\de^{1/4}.
\label{eq:b2e}
\end{align}
Hence
\begin{equation}
\label{eq:eb2e}
|b^2-eb^2e|_{\ta_M}
\leq
|(1-e)b^2|_{\ta_M}
+
|eb^2(1-e)|_{\ta_M}
<2\sqrt{3}\de^{1/4}.
\end{equation}
Then for $u\in F$, we have
\begin{align*}
\|\La_{\ta_M}(ub')-\La_{\ta_M}(b'u)\|^2
&=
\|\La_{\ta_M}(ub'u^*)-\La_{\ta_M}(b')\|^2
\\
&\leq
|ub'^2u^*-b'^2|_{\ta_M}
\\
&\leq
|u(e-1)b^2eu^*|_{\ta_M}
+
|ub^2(e-1)u^*|_{\ta_M}
\\
&\quad
+
|ub^2u^*-b^2|_{\ta_M}
+
|b^2-eb^2e|_{\ta_M}
\\
&\leq
4\sqrt{3}\de^{1/4}+
|ub^2u^*-b^2|_{\ta_M}
\quad
\mbox{by }(\ref{eq:b2e}),\ (\ref{eq:eb2e}).
\end{align*}
In the above, the second inequality follows from
the Araki--Powers--St\o rmer inequality.
Using again the Araki--Powers--St\o rmer inequality
and (\ref{eq:ubuvph}),
we obtain
\begin{align*}
\|\La_{\ta_M}(ub')-\La_{\ta_M}(b'u)\|^2
&\leq
4\sqrt{3}\de^{1/4}
+
2\|\La_{\ta_M}(ubu^*)-\La_{\ta_M}(b)\|_{\ta_M}
\\
&<
4\sqrt{3}\de^{1/4}+2\de^{1/2}
<10\de^{1/4}<\vep^2.
\end{align*}
Therefore, $b'$ does the job.
\end{proof}

Let $b$ be as in Lemma \ref{lem:norm1vep}.
We will introduce an operator $R_b\col H_{\ta_N}\to H_{\ta_M}$
defined by
\[
R_b (x\xi_{\ta_N})
:=
\La_{\ta_M}(b^{1/2}xb^{1/2})
\quad
\mbox{for }x\in N.
\]
It turns out that $R_b$ is an well-defined bounded operator
in what follows.
Let $x,y\in N$.
Then
\begin{align*}
\langle R_b (x\xi_{\ta_N}), R_b (y\xi_{\ta_N})\rangle
&=
\langle \La_{\ta_M}(b^{1/2}xb^{1/2}),\La_{\ta_M}(b^{1/2}yb^{1/2})\rangle
\\
&=
\ta_M(b^{1/2}y^*bxb^{1/2})
=
\ta_M(by^*bx)
\\
&=
\langle\La_{\ta_M}(bx),\La_{\ta_M}(yb)\rangle.
\end{align*}
We have
\begin{align*}
\|\La_{\ta_M}(bx)\|^2
&=
\ta_M(x^*b^2x)
=
\ta_N(x^*T(b^2)x)
\\
&\leq
(1+\vep)
\ta_N(x^*x).
\end{align*}
Thus
$\|\La_{\ta_M}(bx)\|\leq(1+\vep)^{1/2}\|x\xi_{\ta_N}\|$.
Similarly,
$\|\La_{\ta_M}(yb)\|\leq (1+\vep)^{1/2}\|y\xi_{\ta_N}\|$.
Hence
\[
|\langle R_b (x\xi_{\ta_N}), R_b (y\xi_{\ta_N})\rangle|
\leq
(1+\vep)\|x\xi_{\ta_N}\|\|y\xi_{\ta_N}\|.
\]
This shows that $\|R_b\|\leq(1+\vep)^{1/2}$.

\blem\label{lem:R_b}
Let $\vep>0$ and $F\subs N$ be as before.
Take $b$ and $e$ as in Lemma \ref{lem:norm1vep}.
Let $R_b$ be the associated operator defined above.
Then the following statements hold:
\benu
\item
$R_b$ is a c.p.\ operator from $H_{\ta_N}$
into $H_{\ta_M}$;

\item
One has
\[
|\langle R_b^*R_b (x\xi_{\ta_N}), y\xi_{\ta_N}\rangle
-
\langle x\xi_{\ta_N}, y\xi_{\ta_N}\rangle|
<\vep\|y\|+2\vep\|x\|\|y\|
\]
for all $x\in F$ and $y\in N$.
\eenu
\elem

\begin{proof}
(1).
This is trivial.

(2).
Since
\[
\|\La_{\ta_M}(yb)\|
=\ta_M(by^*yb)^{1/2}
\leq\|y\|\ta_M(b^2)^{1/2}
\leq\|y\|,
\]
we have
\begin{align*}
|\langle R_b^*R_b (x\xi_{\ta_N}), y\xi_{\ta_N}\rangle
-
\langle \La_{\ta_M}(xb),\La_{\ta_M}(yb)\rangle|
&=
|\langle \La_{\ta_M}(bx)-\La_{\ta_M}(xb),\La_{\ta_M}(yb)\rangle|
\\
&\leq
\|\La_{\ta_M}(bx)-\La_{\ta_M}(xb)\|\|\La_{\ta_M}(yb)\|
\\
&\leq
\vep\|y\|,
\end{align*}
and
\begin{align*}
&|\langle \La_{\ta_M}(xb),\La_{\ta_M}(yb)\rangle
-\langle x\xi_{\ta_N},y\xi_{\ta_N}\rangle|
\\
&=
|\ta_M(by^*xb)-\ta_N(y^*x)|
=
|\ta_M(y^*xb^2)-\ta_N(y^*x)|
\\
&=
|\ta_N(y^*x(T(b^2)-1))|
\\
&\leq
|\ta_N(y^*x(T(b^2)-e))|
+
|\ta_N(y^*x(e-1))|
\\
&\leq
\|x\|\|y\|\|T(b^2)-e\|
+
\|x\|\|y\|\ta_N(1-e)
\\
&\leq
2\vep\|x\|\|y\|.
\end{align*}
Hence we are done.
\end{proof}

\begin{proof}[Proof of Theorem \ref{thm:norm1proj}]
We have assumed that $N$ is finite and $M$ is semifinite.
Let $\ta_N,\ta_M$ and $T$ be as before.
Our proof presented here is similar to that of Theorem \ref{thm:cross}.

Let $\cF$ be the collection of all finite subsets 
in the norm unit ball of $N$.
Then $\cF$ is a directed set as before.
Applying Lemma \ref{lem:norm1vep}
and Lemma \ref{lem:R_b} to $\vep$ and $F\in \cF$,
we obtain $b(\vep,F)\in n_{\ta_M}\cap M_+$ such that
\[
|\langle R_{b(\vep,F)}^*R_{b(\vep,F)}
(x\xi_{\ta_N}), y\xi_{\ta_N}\rangle
-
\langle x\xi_{\ta_N},y\xi_{\ta_N}\rangle|
<3\vep
\quad
\mbox{for all }x,y\in F.
\]

Since $M$ has the HAP,
there exists a c.c.p.\ compact operators $T_{(\vep,F)}$ on $H_{\ta_M}$
such that 
\[
|\langle R_{b(\vep,F)}^*T_{(\vep,F)}R_{b(\vep,F)}
(x\xi_{\ta_N}), y\xi_{\ta_N}\rangle
-
\langle R_{b(\vep,F)}^*R_{b(\vep,F)}
(x\xi_{\ta_N}), y\xi_{\ta_N}\rangle
|
<\vep
\]
for all $x,y\in F$.
If we set $U_{(\vep,F)}:=R_{b(\vep,F)}^*T_{(\vep,F)}R_{b(\vep,F)}$,
then $U_{(\vep,F)}$ is a c.p.\ compact operator
on $H_{\ta_N}$, because $T_{(\vep,F)}$ is compact.
It turns out that $U_{(\vep,F)}$ converges to 1 weakly
from the fact that
$\|U_{(\vep,F)}\|\leq\|R_{b(\vep,F)}\|^2\leq1+\vep$
and
\[
|\langle U_{(\vep,F)} (x\xi_{\ta_N}), y\xi_{\ta_N}\rangle
-\langle x\xi_{\ta_N},y\xi_{\ta_N}\rangle|
<4\vep
\quad
\mbox{for all }x,y\in F.
\]
Then the net $(1+\vep)^{-1}U_{(\vep,F)}$ does the job.
\end{proof}

Let $G$ be a locally compact quantum group
in the sense of \cite{kv}.
Roughly speaking, $G$ consists of
a von Neumann algebra $L^\infty(G)$
and
a coproduct
$\De\col L^\infty(G)\to L^\infty(G)\oti  L^\infty(G)$.
Then $G$ is said to be \emph{amenable}
if there exists a state $m$ on $L^\infty(G)$,
which is called an invariant mean on $G$,
such that $(\id\oti m)\circ\De(x)=m(x)$.

Let $\al$ be an action of $G$ on a von Neumann algebra.
Namely, $\al$ is a unital faithful normal
$*$-homomorphism from $M$ into $M\oti L^\infty(G)$
such that
$(\al\oti\id)\circ\al=(\id\oti\De)\circ\al$.
If $m$ is an invariant mean on $G$,
then the map $(\id\oti m)\circ\al$
is a norm one projection from $M$ onto
$M^\al:=\{x\in M\mid \al(x)=x\oti1\}$,
the fixed point algebra.
Thus the following result is an immediate consequence
of Theorem \ref{thm:norm1proj}.

\bcor
Let $G$ be an amenable locally compact quantum group.
Let $\al$ be an action of $G$ on a von Neumann algebra $M$.
If $M$ has the HAP,
then the fixed point algebra $M^\al$ has the HAP.
\ecor

By the duality argument,
we can generalize Theorem \ref{thm:cross}
as follows.

\bcor
Let $G$ be a locally compact quantum group
whose dual quantum group is amenable.
Let $\al$ be an action of $G$ on a von Neumann
algebra $M$.
If $M\rti_\al G$ has the HAP,
then so does $M$.
\ecor

Finally, we present a generalization of
Corollary \ref{cor:abelian}
that is obtained from the previous corollary
and the fact that $M\rti_\al G$
equals the fixed point algebra of $M\oti \bB(L^2(G))$
by a $G$-action.

\bcor
Let $G$ be an amenable locally compact quantum group
whose dual quantum group is also amenable.
Let $\al$ be an action of $G$ on a von Neumann
algebra $M$.
Then
$M\rti_\al G$ has the HAP
if and only if so does $M$.
\ecor

\begin{remark}
If we apply the same proof of Theorem \ref{thm:norm1proj}
to the inclusion
$N\subs M$
such that $M$ is semidiscrete,
then we can show that
$N$ is semidiscrete.
In particular,
this gives a proof of the fact that
the injectivity implies the semidiscreteness.
Indeed,
let $M$ be an injective von Neumann algebra
which is acting on a Hilbert space $H$.
Then we have a norm one projection $E$ from $\bB(H)$
onto $M$.
Since $\bB(H)$ is semidiscrete,
$M$ is semidiscrete.
\end{remark}


\end{document}